\documentclass[12pt]{article}

\usepackage{amsmath,amsfonts}
\usepackage{amsthm}
\usepackage{amssymb}
\usepackage{hyperref}
\hypersetup{
    colorlinks=true,
    linkcolor=blue,
    filecolor=magenta,
    urlcolor=cyan,
}
\usepackage{mathrsfs}
\usepackage{tikz}
\usetikzlibrary{automata, positioning, arrows.meta}
\usepackage[normalem]{ulem}

\usepackage{caption}
\newtheorem{theorem}{Theorem}

\newtheorem{claim}{Claim}

\newtheorem{conj}{Conjecture}

\newtheorem{lemma}{Lemma}

\newtheorem{definition}{Definition}

\newtheorem{proposition}{Proposition}

\newtheorem{corollary}{Corollary}
\newtheorem{question}{Question}

\theoremstyle{definition}
\newtheorem{example}{Example}

\def\a{\alpha}
\def\b{\beta}

\def\s{\sigma}

\def\l{\ell}

\def\H{\mathcal{H}}
\def\G{\Gamma}
\def\N{\mathbb{N}}

\def\F{\mathcal{F}}
\def\S{\Sigma}
\def\W{\mathcal{W}}
\def\V{\mathcal{V}}

\def\C{\mathscr{C}}
\def\dk{\Delta(K)}

\begin{document}

\title{Value of Information in Dynamic Decision Making%
\thanks{The authors thank Ehud Lehrer and Abraham Neyman for their valuable comments, suggestions, and overall support with the current work.
Shaiderman acknowledges the support of the Israel Science Foundation, Grant \#992/23.
Solan acknowledges the support of the Israel Science Foundation, Grant \#211/22.
}
}

\author{Dimitry Shaiderman%
\thanks{Department of Mathematics, Hebrew University of Jerusalem, Jerusalem 9190401, Israel. e-mail: dima.shaiderman@gmail.com.}
and
Eilon Solan%
\thanks{School of Mathematical Sciences, 
Tel Aviv University, 6997800, Israel.
e-mail: eilonsolan@gmail.com.}}

\date{\today}

\maketitle

\begin{abstract}
We study the value of information in predicting the evolving state of a finite Markov chain.
At each stage, a decision maker chooses a state and observes only whether the current state of the chain matches her choice; the resulting information is used to make a prediction on the state at the final stage of the problem.
We show that, when the chain starts from an invariant distribution, the optimal terminal value is non-decreasing with the number of observations and converges at a uniform exponential rate. We introduce the predictive learning index, which measures whether all attainable informational value is extracted after finitely many observations, and show that both finite and infinite indices may occur. In contrast, for nonstationary initial distributions, the value may strictly decrease with the horizon.
\end{abstract}

\noindent\textbf{Subject Classification:}
Decision analysis: Theory:
Value of Information in Dynamic Decision Making.

Dynamic programming: Markov/Finite State:
Value of Information in Dynamic Decision Making.

\section{Introduction}\label{Sec_Intro}

Classical utility theory asks how a decision maker should act when the state of nature is unknown. 
When information about the state arrives gradually, a simple principle holds: additional information can only benefit the decision maker;
that is, her optimal expected payoff can only rise as more information is revealed. This is the spirit of Blackwell's classical work comparing sources of information (Blackwell, \cite{Blackwell:1},\cite{Blackwell:2}): a more informative signal is always at least as valuable as a less informative one, regardless of the decision problem at hand.

Does the same logic apply when the underlying state of nature changes dynamically over time?
In this paper we study this problem when the decision maker has to make a decision at a given future time, and receives information on the state of nature as times goes by.
Such a situation arises naturally in search, surveillance, and screening problems. 
For example, a naval unit may search for a submarine, 
which changes its location daily,
with the goal of identifying its location at a given future date.
Similarly, hospitals may screen patients for a specific pathogen, to ensure that they can undertake a certain treatment at a certain date.%
\footnote{In a hospital setting, which pathogen a patient carries can change over time through contact with other patients or the environment, quite apart from any testing, and is modeled as a Markov chain
(Cooper and Lipsitch, 2004). 
Routine surveillance swabs monitor this process without influencing it.}

In both examples above, the information that the decision maker gains in every period is a yes/no signal;
that is, whether the submarine is present in a given sector when the search is done, 
or whether the patient carries a certain pathogen when the test is done.

Should the principle of positive value of information hold in the dynamic setup as well? That is, 
\color{black}
is the marginal informational value of each new observation positive? 
Or does the information saturate from a certain period on? 
These are the questions that we study in the present paper.
As we will see, the principle of positive value of information sometimes fails.
The reason is that with each period that goes by, two things change simultaneously: the decision maker has an additional observation opportunity, but the payoff-relevant state is also one transition farther into the future.

\paragraph{The model.}
We consider a discrete-time Markov chain $\{X_j\}_{j\geq 1}$ over a finite state space $K=\{1,\dots ,k\}$, which may be described by a prior distribution $q$ over $K$ and a $k\times k$ dimensional stochastic matrix $M$, determining the transition rule of the chain. 

A decision maker is engaged in the following problem with horizon $n$:
at each stage $j=1,\dots,n$ the decision maker selects
an \emph{action} $\xi_j \in K$, and then observes the binary signal $I (\xi_j) := \textbf{1}\{X_j = \xi_j\}$; i.e., the decision maker is told whether her $j$'th stage action matches the realized state at that stage. 
This action is allowed to depend on all information available to the decision maker at the start of the $j$'th stage, which consists of both her own past actions $\xi_1,\dots,\xi_{j-1}$, as well as on $I(\xi_1),\dots,I(\xi_{j-1})$. Let $\Xi$ be the space of all policies or strategies available to the decision maker.

The decision maker's goal is to maximize the utility $u(X_n,\xi_n)$, which depends both on the actual state $X_n$ at stage $n$
as well as on the decision maker's action at that stage $\xi_n$,
where $u$ is some utility function.
Thus, while $\xi_1,\dots,\xi_{n-1}$ provide \emph{information},
$\xi_n$ can be thought of as a \emph{prediction} of the state that maximizes the expected utility at stage $n$.

The \emph{value} of the problem,
that is, the maximal expected amount attainable by the decision maker, is 
$$v_n (q) := \max_{\xi \in \Xi} E_{q,\xi}\,\left[u(X_n,\xi_n)\right],$$
where $E_{q,\xi}$ is the expectation operator induced by the prior $q$, transition rule $M$, and the strategy $\xi$, on the sequence $(X_1,\xi_1,\dots ,X_n,\xi_n)$. 

In search problems where $u(X_n,\xi_n) = 1\{X_n = \xi_n\}$,
the decision maker's goal is to correctly identify the target's location at stage $n$.
In the setup of search described above, $v_n(q)$ is the best attainable probability of correctly locating the submarine on day $n$ given $n-1$ days of prior tracking.
Similarly, in the hospital screening problem, $v_n(q)$ is the best attainable probability of correctly determining whether a patient carries a given pathogen on the day of a scheduled procedure, given $n-1$ preceding days of screening. 

\paragraph{The object we are interested in.}
Our work centers around the study of properties of the sequence $\{v_n(q):n\geq 1\}$, for a fixed prior $q$. When this sequence is increasing, 
the value of information grows at any step;
when it is non-decreasing, the decision maker does not lose from obtaining more information;
and when it is eventually constant learning is finite, in the sense that the decision maker does not gain from additional observations past a certain time period.
When the sequence 
$\{v_n(q):n\geq 1\}$ is non-decreasing, its rate of convergence   quantifies the speed by which the informational value of future observations diminishes.

\paragraph{Main Results.} Our first findings concern stationary Markov chains, i.e., chains in which the prior is an invariant distribution $\pi$ of the stochastic matrix $M$.  Theorem \ref{Thm1} in Subsection \ref{Subsec:Stationary_Results} reveals that the sequence $\{v_n(\pi):n\geq 1\}$ is non-decreasing, and thus converges to a limit denoted $v_{\infty}(\pi)$. 
In informational terms, when the chain is stationary, the decision maker 
never loses from additional observations.
In addition, Theorem \ref{Thm1} establishes that the convergence occurs at a geometric rate, with an explicit geometric bound depending only on the cardinality of the state space. 

To study the timeline of learning we introduce in Subsection \ref{Subsec:Main_Index} the notion of the \textit{predictive learning index}, denoted $i(\pi,M)$. 
Such index counts the number of steps required for $v_n(\pi)$ to attain $v_{\infty}(\pi)$. That is, the decision maker can extract all payoff relevant information using her first $i(\pi,M)-1$ observations.

Theorem \ref{Thm2} shows that the predictive learning index $i(\pi,M)$ equals to the minimal $n \in \N \cup \{+\infty\}$ such that $v_n(\pi) = v_{n+1}(\pi)$. In particular, if $i(\pi,M)=+{\infty}$, then the value of information must grow at every step. Example \ref{Example:2} in Subsection \ref{Subsec:Main_Index} illustrates that such a phenomenon, which corresponds to infinite learning, is possible. 
The complementary phenomenon of finite learning, i.e., $i(\pi,M)<+{\infty}$, occurs in several common classes of Markov chains, 
e.g., i.i.d.~chains and deterministic chains.
\color{black}
Example \ref{Example:1} utilizes Theorem \ref{Thm2} to show that finite learning can occur in less common scenarios as well. 

In the non-stationary case, the value of information need no longer grow with time. 
In fact, Example \ref{Example:3} in Subsection \ref{Subsec_Results_General} shows that $\{v_n(q):n\geq 1\}$ can be strictly decreasing for a specified prior $q$, stochastic matrix $M$, and utility $u$.
In particular, the value of information can be negative at every step. 
Proposition \ref{Prop.1} shows that this failure of monotonicity is not confined to the preceding example:
for any $n_1<n_2\in \N$, there exists a prior distribution $q$ such that $v_{n_2}(q)\leq v_{n_1}(q)$. 
Thus, away from the stationary distribution, the benefit of having an additional opportunity to learn can be outweighed by the loss of predictability caused by the Markov chains stochastic dynamics.

Lastly,  Theorem \ref{Thm3} in Subsection \ref{Subsec_Results_General} shows that when the Markov chain is ergodic, 
irrespective of the prior $q$, the sequence $\{v_n (q):n\geq 1\}$ converges to $v_{\infty}(\pi)$ at a geometric rate, 
uniformly over all priors.%
\footnote{For ergodic Markov chains, i.e., irreducible and aperiodic, there exists a unique invariant distribution $\pi$.}

\paragraph{Managerial Insights.}
From an operational perspective, our results address the question of how far in advance of a terminal decision information acquisition should begin. 
In stationary environments, starting earlier cannot reduce the optimal terminal payoff, while the additional benefit available from extending the information-acquisition window decays geometrically. The predictive learning index identifies environments in which there is a finite look-back horizon beyond which earlier observations have no additional value. 
In contrast, in non-stationary environments,
a longer information-acquisition horizon need not improve terminal performance, highlighting the importance of distinguishing steady-state operations from transient regimes.

\paragraph{Related Literature.} Our model is closely related to the literature on surveillance search for moving targets. Tierney and Kadane \cite{Tierney-Kadane} study a target evolving according to a Markov process and a searcher who allocates search effort over time and space. Similarly to our setup, intermediate detections need not terminate the search and may instead provide information useful for subsequent surveillance. 
The main difference between the two papers is the focus.
Following earlier works on the search for moving targets by Brown \cite{Brown} and Stone and Kadane \cite{Stone-Kadane}, Tierney and Kadane \cite{Tierney-Kadane} 
focus on finding optimal search plans for a given surveillance problem. Our object of interest is how the optimal terminal value changes with the horizon.  

The task of predicting the next step of a Markov chain using previously acquired information is central in several game theoretic models. 
Renault \cite{Renault} and Gensbittel and Renault \cite{Gensbittel},
see also H\"orner et al.~\cite{Horner},
study repeated zero-sum games in which payoff-relevant private states evolve according to Markov chains. Information about the evolving state is strategically revealed through players' actions, and the object is the long-run value of the game. 
In contrast, our model has a single decision maker, information is generated mechanically by the outcomes of her own actions, and our object is the value of this information for a terminal decision.

Our model is also related to active information acquisition and controlled sensing (Chernoff \cite{Chernoff}; Naghshvar and Javidi \cite{Naghshvar}; Nitinawarat, Atia, and Veeravalli \cite{Nitinawarat}), where actions determine the information observed, but the unknown hypothesis is typically fixed rather than evolving over time as in our setting.

The findings and analysis of the current work 
are partially motivated
by results and techniques in Lehrer and Shaiderman \cite{MPwSR} and Shaiderman \cite{Monotonicity}, 
which 
study
the analytic behavior of the value functions in a dynamic Bayesian persuasion model and Markov chain games.

\paragraph{Organization of the Paper.} The main results are provided in Section \ref{Sec:Main_Results}. Section \ref{Sec:MDP} is devoted to the introduction of a specialized Markov decision problem, which will serve as the framework for the analysis in the work. The proofs of the main results appear in Section \ref{Sec:Proofs}. The paper is concluded with a discussion and survey of additional open problems (see Section \ref{Sec_Discussion}). 

\paragraph{Notations.} For every finite set $A$, $\Delta(A)$ denotes the set of probability measures over $A$. For every $p \in \Delta(A)$ and $a \in A$, $p(a)$ denotes the probability mass assigned by $p$ to $a$. For every $a \in A$, $\delta_a \in \Delta(A)$ denotes the Dirac measure supported on $a$.
For a matrix $M$ of dimensions $\vert A \vert \times \vert A \vert$ and $p \in \Delta(A)$,  $pM$ denotes the $M$\textit{-shift} of $p$; i.e., the matrix multiplication of $p$ with $M$, where $p$ is viewed as an $\vert A \vert$-dimensional row vector. Lastly, $\log$ denotes the natural logarithm. 
\section{Main Results}\label{Sec:Main_Results}

\subsection{Stationary Markov Chains}\label{Subsec:Stationary_Results}

We first focus on stationary Markov chains; namely, chains whose prior is an invariant distribution $\pi$ of the transition matrix $M$. Our first main result, stated in Theorem \ref{Thm1}, establishes that starting from any invariant distribution $\pi$, the value of information grows with time. The Theorem \ref{Thm1}  also provides an upper bound for the speed at which the value of new information perishes with time.

\begin{theorem}\label{Thm1}
For any stationary Markov chain $\{X_n\}_{n\geq 1}$, the sequence $\{v_{n}(\pi):n\geq 1\}$ is non-decreasing, and converges to a scalar $v_{\infty}(\pi) \in [0,1]$, such that for every $n\geq 1$,
  \begin{align}\label{Eq. Thm1}
     0\leq  v_{\infty}(\pi) - v_{n}(\pi) \leq \left(1-\frac{1}{k}\right)^{n-1}.
  \end{align}
\end{theorem}

Theorem \ref{Thm1} holds irrespective of the underlying structure of the Markov chain (irreducible, periodic, reversible, etc.). Theorem \ref{Thm1}
implies that for any error term $\varepsilon>0$, the decision maker can approximate $v_{\infty}(\pi)$ up to $\varepsilon$ using the information obtained from the first $\lceil \log_{(1-1/k)} \varepsilon \rceil$ observations.%
\footnote{Note that the maximal expected utility available to the decision maker after the first $n$ observations is $v_{n+1}$.}
From an asymptotic perspective, i.e., when $k$ is large, the information obtained from the first $k\cdot\lfloor \log(k) \rfloor$ observations suffices to approximate $v_{\infty}(\pi)$ up to an error term of $1/k$; indeed, by Theorem \ref{Thm1},
\begin{align*}
0\leq v_{\infty}(\pi) - v_{\lfloor \log(k) \rfloor \cdot k + 1}(\pi) & \leq \left(1-\frac{1}{k}\right)^{\lfloor \log(k) \rfloor\cdot k}\\
& \sim  \exp(-\log (k))=\frac{1}{k}.
\end{align*}

In view of Theorem \ref{Thm1}, several questions arise. First, is the bound given in Equation \eqref{Eq. Thm1} sharp? Another question concerns the timeline of the convergence of $\{v_{n}(\pi):n\geq 1\}$; can the convergence occur in finite time, or alternatively, can the sequence $\{v_{n}(\pi):n\geq 1\}$ be strictly increasing. 

While the sharpness of the bound in \eqref{Eq. Thm1} remains an open problem, the study of the timeline of convergence towards $v_{\infty}(\pi)$ admits several results and raises additional questions. For that reason we devote the proceeding subsection to the introduction and study of a new notion, titled the \textit{predictive learning index}.   

\subsubsection{The Predictive Learning Index}\label{Subsec:Main_Index}

\begin{definition}
For any invariant distribution $\pi$ of the Markov chain $\{X_j\}_{j\geq 1}$ with transition rule given by the stochastic matrix $M$, the \emph{predictive learning index}, denoted $i(\pi,M)$, is defined by
\begin{align*}
\displaystyle i(\pi,M) =  \min \lbrace n \in \N\cup \{+ \infty\} : v_n(\pi) = v_{\infty}(\pi) \rbrace.
\end{align*}
\end{definition}

From an informational stand point, since by Theorem \ref{Thm1} $\{v_{n}(\pi):n\geq 1\}$ is non-decreasing, whenever $i(\pi,M)=+\infty$, the value of information can be viewed as growing at every step. On the opposite, when $i(\pi,M)$ is finite, the gathering of new information has no value from some time on. In such a case, the predictive learning index describes the minimal amount of time required to extract the information with which the decision maker can guarantee $v_{\infty}(\pi)$.  

Several basic examples where the predictive learning index is finite include:
\begin{itemize}
\item If $\{X_j\}_{j\geq 1}$ forms an i.i.d.\ sequence of random variable (i.e., all rows of $M$ equal to $\pi$), then $i(\pi,M) = 1$. Indeed, in such a case, the distribution of the next step of the chain is independent of any past information, implying that
\begin{align*}
    v_n(\pi) = \max_{\l \in K} \sum_{i=1}^k \pi(i) \cdot u(i,\l), \quad \forall n\geq 1.
\end{align*}
\item If $\{X_j\}_{j\geq 1}$ is a binary Markov chain (i.e., $K=\{1,2\}$), then the sequence of outcomes $\{I(\xi_j): j=1,...,n\}$ provides full information regarding $X_1,...,X_n$, regardless of the strategy $\xi \in \Xi$ the decision maker chooses to follow. Therefore, in such case, starting from the prior $\pi$, for every $n\geq 2$,
\begin{align*}
    v_n (\pi) & =  \pi (1)\cdot \max_{\l \in K}\sum_{i \in K} (\delta_{\{1\}}M) (i) \cdot u(i,\l)\\
    & \quad  + \pi(2)\cdot \max_{\l \in K}\sum_{i \in K} (\delta_{\{2\}}M) (i) \cdot u(i,\l),
\end{align*}
implying that $i(\pi,M) \leq 2$.
\item If $\{X_j\}_{j\geq 1}$ forms a deterministic sequence conditional on $X_1$ (i.e., the matrix $M$ is a permutation matrix), then $i(\pi,M) \leq \# \text{support}(\pi)$. There exists a deterministic sequence of actions that assures that the decision maker would know the realized state of the $(\# \text{support}(\pi)-1)$'st step of the chain after observing the signals of her first $(\# \text{support}(\pi)-1)$'st actions. By the deterministic nature of the chain, such knowledge suffices to know the realized state at any step proceeding the $(\# \text{support}(\pi)-1)$'st step. For that reason,
\begin{align*}
     v_n(\pi) & = \sum_{i=1}^k P(X_n=i) \cdot\max_{\l \in K}   u(i,\l)\\
      & = \sum_{i=1}^k \pi
      (i) \cdot\max_{\l \in K}   u(i,\l), \quad \forall n\geq \#\text{support}(\pi).
\end{align*}
\end{itemize}

The next result provides a sufficient condition for the predictive learning index to be finite. 
\begin{theorem}\label{Thm2}
    If $v_{n+1}(\pi) = v_n(\pi) $ for some $n\geq 1$, then $v_{n+1+j} (\pi) = v_n(\pi)$ for all $j \in \N$. In particular,  
    \begin{align*}
        \displaystyle i(\pi,M) =  \min \lbrace n \in \N \cup \{+\infty\} :  v_{n+1} (\pi) = v_n(\pi)   \rbrace.
    \end{align*}
\end{theorem}


Theorem \ref{Thm2} implies that at the first stage $n\geq 1$ at which $v_n(\pi) = v_{n+1}(\pi)$, there is no more value for the information obtained along proceeding stages. From a practical stand point, Theorem \ref{Thm2} allows to provide additional instances where the predictive learning index is finite. To showcase such an instance, introduce the \textit{matching utility} $u_*:K\times K \to \{0,1\}$ defined by
\begin{align*}
    u_*(i,\l) := \textbf{1}\{i=\l\}, \quad \forall (i,\l) \in K \times K.
\end{align*}
Under such a utility, $v_n (q) = \max_{\xi \in \Xi} P_{q,\xi}(X_n=\xi_n)$. Thus, the $n$'th stage optimization problem faced by the decision maker, is the one where she seeks to choose her first $n-1$ actions, so that the obtained information will allow her to maximize the probability of correctly matching the state of the chain in stage $n$.

The following example provides a Markov chain, which is not i.i.d.\ nor conditionally deterministic given $X_1$, and has a finite index for the matching utility. 

\begin{example}\label{Example:1}
    Consider the stochastic matrix:
\begin{align*}
    M^* = \begin{bmatrix}
0.424 & 0.128 & 0.192 & 0.256 \\
0.064 & 0.488 & 0.192 & 0.256 \\
0.064 & 0.128 & 0.552 & 0.256 \\
0.064 & 0.128 & 0.192 & 0.616
\end{bmatrix},
\end{align*}
for which the unique stationary distribution is $\pi = (0.1,0.2,0.3,0.4)$ . For such an example, under the matching utility, computer simulations show that $v_2 (\pi) = 0.4888 < 0.5000768 = v_3(\pi)=v_4(\pi)$. Hence, Theorem \ref{Thm2} implies that for this specific example $i(\pi,M^*)=3$.
\end{example}
In Section \ref{Sec_Discussion} we revisit the Markov chain in Example \ref{Example:1}, and discuss another possible class of Markov chains conjectured to have a finite index for the matching utility.

The next proposition establishes that the case where $i(\pi,M)=+\infty$ is also possible. As a preparation for its statement, we note that whenever a matrix is doubly stochastic, its unique invariant distribution is the uniform distribution over $K$, i.e., $\pi(\l) = 1/k$ for every $\l\in K$.

\begin{proposition}
    \label{Prop.1}
    For the matching utility $u_*$, there exists a doubly stochastic matrix $M$ such that $i(\pi,M) = +\infty$.
\end{proposition}

Let us describe the example which is the cornerstone for the above proposition, and provide the key heuristics leading to the phenomenon of infinite index in the example.

\begin{example}\label{Example:2}
Consider a simple random walk on a circle having four states, as depicted in the form of a Markov chain in Figure \ref{fig:markov_clock}. In this Markov chain, the states alternates between $\{1,3\}$ and $\{2,4\}$. If say, at stage $n$ the state is 2, then at stage $n+1$ it is $1$ or $3$ with equal probabilities. For such a chain the unique invariant distribution is $\pi = (1/4,1/4,1/4,1/4)$. From a learning angle, such Markov chain contains two key properties:

\begin{itemize}
    \item[(a)] \textit{Full learning upon first positive signal.} After the first positive signal,   the decision maker gleans all possible information: she knows whether the states $\{1,3\}$ occur in even stages or in odd stages.  
    
    \item[(b)] \textit{Learning does not end before the first positive signal.} Starting from $\pi$, as long as no positive signal occurred, for each $i \in \{1,2,3,4\}$, the likelihood of the next step of the Markov chain being equal to $i$ is strictly lower then $1/2$. Therefore, any prediction of the decision maker regarding the next step will match the realized state with probability strictly less then $1/2$.     
\end{itemize}

Based on the above two properties, the center of attention shifts to the likelihood of having a streak of $n$ negative signals. Here two phenomena arise as well. Starting from $\pi$, for any $n\geq 1$ and for any strategy $\xi \in \Xi$: (i) there is positive probability for a streak of $n$ negative signals, and (ii) such probability vanishes as $n \to \infty$.\footnote{While (i) is an intrinsic property of the described Markov chain, (ii) may be obtained for any ergodic Markov chain, as can be seen in the discussion culminating with relation \eqref{Bound on W_n} in Subsection \ref{Subsec_Results_General}.} Properties (a) and (b) described above together with (i) suffice to show that $v_n(\pi)<1/2$ for every $n\geq 1$. Property (a) together with (ii) imply that $v_n(\pi) \to 1/2$ as $n\to \infty$, so that indeed $i(\pi,M) =+\infty$. The exact details, with slightly modified arguments and suitable notations, can be found in Subsection \ref{Subsec_Proofs_Prop1}.  
\end{example}

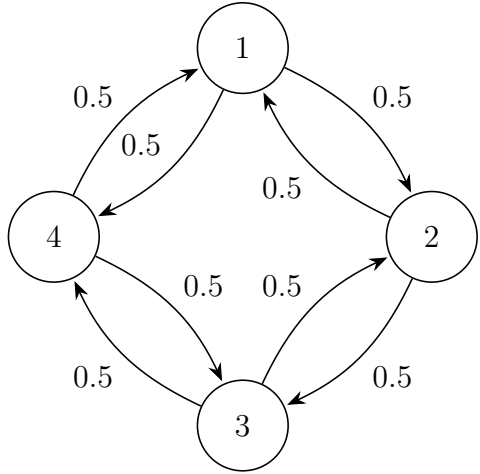
\begin{figure}[htpb]
    \centering
    \begin{tikzpicture}[
        > = {Stealth[scale=1.2]}, 
        shorten > = 1pt,          
        auto,
        semithick,
        state/.style={circle, draw, minimum size=1.2cm}
    ]

    \node[state] (1) at (90:2.5)  {1};
    \node[state] (2) at (0:2.5)   {2};
    \node[state] (3) at (270:2.5) {3};
    \node[state] (4) at (180:2.5) {4};

    \path[->]
    (1) edge [bend left=20] node {0.5} (2)
        edge [bend left=20] node[swap] {0.5} (4)
        
    (2) edge [bend left=20] node {0.5} (3)
        edge [bend left=20] node {0.5} (1)
        
    (3) edge [bend left=20] node {0.5} (4)
        edge [bend left=20] node {0.5} (2)
        
    (4) edge [bend left=20] node {0.5} (1)
        edge [bend left=20] node {0.5} (3);

    \end{tikzpicture}
    \caption{The Markov chain in Example \ref{Example:2},  representing simple random walk on a 4-state circle.}
    \label{fig:markov_clock}
\end{figure}

By the Birkhoff–von Neumann Theorem \cite{Birkhoff}, every doubly stochastic matrix can be decomposed into a convex combination of permutation matrices. Since the predictive learning index of permutation matrices is finite, this, together with Proposition \ref{Prop.1}, implies that for the matching utility $u_*$ the set  
\begin{align*}
    \mathcal{M}^{i\in \N}_{\pi} :=\{M \in \mathcal{M}_{\pi}: i(\pi,M)<+\infty\},
\end{align*}
is not convex, where $\mathcal{M}_{\pi}$ denotes the set of all $k\times k$ dimensional stochastic matrices with invariant distribution $\pi$.

\subsection{General Prior}\label{Subsec_Results_General}

Let us now shift our attention to the case where the prior of the Markov chain $\{X_j\}_{j\geq 1}$ is not necessarily an invariant distribution. 
Our first finding exhibits a Markov chain in which the value of information strictly decreases at every step, when the initial distribution is not invariant.

\begin{example}\label{Example:3}
    Consider the binary Markov chain over $K=\{1,2\}$ whose transition rule is described in Figure \ref{fig:binary_chain}, starting from the prior $q_* = (2/3,1/3)$. Let $u=u_*$ be the matching utility. For $n=1$, state $1$ has a higher likelihood under $q_*$, and thus $v_1(q_*) = 2/3$.
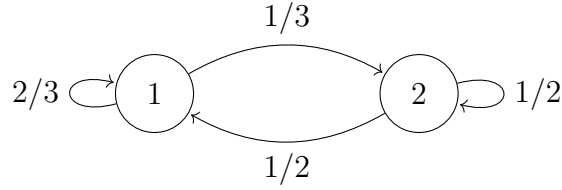
\begin{figure}[htpb]
    \centering
\begin{tikzpicture}[shorten >=1pt, node distance=3.5cm, on grid, auto]

  \node[state] (1) {$1$};
  \node[state] (2) [right=of 1] {$2$};

  \path[->]
    (1) edge [loop left]  node {$2/3$} (1)
        edge [bend left]  node {$1/3$} (2)
    (2) edge [loop right] node {$1/2$} (2)
        edge [bend left]  node {$1/2$} (1);

\end{tikzpicture}

\caption{The Binary Markov chain in Example \ref{Example:3}, with decreasing value of information starting from the prior $(2/3,1/3)$.}
    \label{fig:binary_chain}
\end{figure}
As discussed in Subsection \ref{Subsec:Main_Index}, since the chain is binary, the past learning is complete: for every $n\geq 2$, $X_1,...,X_{n-1}$ are known upon observing the first $n-1$ signals, irrespective of the strategy $\xi \in \Xi$. Therefore, for each $n\geq 2$, 
\begin{align}\label{Eq:example2}
v_n(q_*) = P_{q_*}(X_{n-1} =1)\cdot 2/3 + (1-P_{q_*}(X_{n-1}=1))\cdot 1/2.
\end{align} 
Define $I_1:= q_*(1)=2/3$, $I_n:= P_{q_*}(X_{n-1}=1)$ for every $n\geq 2$. As $I_n = 2/3\cdot I_{n-1} + 1/2\cdot (1-I_{n-1})$ for every $n\geq 2$, relation \eqref{Eq:example2}  implies that $v_n(q_*) = I_n$ for every $n\geq 2$. By a recursive computation,  $I_n = 6/10 + 1/15\cdot (1/6)^{n-1}$ for every $n\geq 2$, and thus $v_n(q_*) =  6/10 + 1/15\cdot (1/6)^{n-1}$ for $n\geq 2$ as well. Therefore, the sequence $\{v_n(q_*):n\geq 1\}$ is strictly decreasing. 
If the chain is stationary, i.e., starts from the unique invariant distribution is $\pi = (6/10,4/10)$, then it can be verified that $v_n(\pi) = 6/10$ for every $n\geq 1$. 
\end{example}

In view of Example \ref{Example:3} a question arises as to whether for any Markov chain there always exists a prior at which the value of additional information is negative. While this remains an open question, our next result establishes a weaker result, concerning the averages of values at certain priors. 

 To state such a result, denote by $m_1,...,m_k \in \dk$ the rows of the stochastic matrix $M$, and consider for each $n\geq 1$ the number 
\begin{align*}
    \Phi_n (\pi) := \pi(1) \cdot v_n(m_1) + \cdots + \pi(k)\cdot v_n(m_k).
\end{align*}

where $\pi$ is a invariant distribution of $M$. For each $i\in K$, $m_i$ describes the distribution of the next step of the Markov chain, given that it is currently located at state $i$. Thus, the values $\{v_n(m_i): i \in K\}$ may be viewed as those capturing the scenario in which the decision maker observes the realized state of a preceding step $X_0$, and then engages in the $n$'th stage problem with $X_1,...,X_n$.   

\begin{proposition}\label{Prop3}
    The sequence $\{\Phi_n (\pi): n\geq 1\}$ is non-increasing.
\end{proposition}

Proposition \ref{Prop3} reveals that the value of information does not always increase with time across all prior distributions. The proposition implies that for any two time periods $n_1<n_2$, there exists a row $m_i$ such that $v_{n_2} (m_i) \leq v_{n_1}(m_i)$. In view of Theorem \ref{Thm1}, Example \ref{Example:3}, and Proposition \ref{Prop3},  the value of information can have opposite nature based on the prior distribution. For some priors (invariant distributions), the decision maker can only gain from additional information, whereas for  other priors, the value of information does not grow at any step.
\bigskip

We conclude this section, by studying whether the value of new information perishes with time for the general case, as previously showcased for stationary chains (e.g., Theorem \ref{Thm1}). Our study revolves around the class of \textit{ergodic} Markov chains, namely, ones whose transition matrix $M$ is both irreducible and aperiodic.

We first list several useful preliminary facts regarding such chains:

\begin{itemize}

\item If $M$ is irreducible and aperiodic, then $\min(M^{k^2})>0$, where  $\min(A)$ denotes the smallest cell of a matrix $A$. This follows from Wielandt's bound \cite{Wielandt}, who proved that $(k-1)^2+1$ is the smallest number satisfying $\min (M^{(k-1)^2+1})>0$ for all irreducible and aperiodic matrices $M$.\footnote{In Wielandt \cite{Wielandt} it is proven that this upper bound is sharp, i.e.,  there exists irreducible and aperiodic $M$ such that $\min (M^{(k-1)^2})=0$. Weilandt's original proof that $\min (M^{(k-1)^2+1})>0$ is published in Schneider \cite{Schneider}. }

\item \textit{The Ergodic Convergence Theorem} (e.g., Theorem 4.9 in Levin and Peres \cite{Peres}): If $M$ is irreducible and aperiodic, then there exist constants $\a \in (0,1)$ and $C>0$ such that
        \begin{align}\label{Ergodic Convergence Thm}
            d_n(M):=\sup_{q\in \dk} \Vert qM^n - \pi\Vert_1 \leq C \a^n,
        \end{align} 
        where $\pi$ is the unique invariant distribution of $M$, and $\Vert \cdot \Vert_1$ denotes the $\l_1$-norm on $\dk$.
\end{itemize}

Define for every $n\geq 1$ and $q\in \dk$:
        \begin{align*}
            \W_n(q,M):= \max_{\xi \in \Xi} P_{q,\xi}(X_1\neq \xi_1,...,X_n\neq \xi_n),
        \end{align*}
        and let $\W_n(M):= \sup_{q\in \dk} \W_n(q,M)$. In words, $\W_n(M)$ describes the highest possible probability of $n$ consecutive negative signals, across all possible prior distributions of $\{X_j\}_{j\geq 1}$, under the transition rule $M$. Since $\min(M^{k^2})\in (0,1)$ , we have  
        \begin{align}\label{Bound on W_n}
            \W_n (M) & \leq \sup_{q\in \dk} \max _{\xi \in \Xi} P_{q,\xi} \left(X_{k^2+1} \neq \xi_{k^2+1},X_{2k^2+1} \neq \xi_{2k^2+1},..., X_{\lfloor \frac{n}{k^2}\rfloor k^2 +1} \neq  \xi_{ \lfloor \frac{n}{k^2}\rfloor k^2 +1} \right) \nonumber  \\
            &\leq \left( 1-\min\left(M^{k^2}\right) \right)^{\lfloor \frac{n-1}{k^2}\rfloor } \to 0 \quad \text{as}\quad n\to \infty.
        \end{align}

\begin{theorem}\label{Thm3}
For any ergodic Markov chain $\{X_j\}_{j\geq 1}$, starting from any prior $q$,  for any $n\geq 2$,
  \begin{align}\label{Thm3_Lower}
     v_n(q) - v_{\infty}(\pi) \geq - \frac{d_{\lceil \frac{n}{2} \rceil} (M)}{\min_{i\in K} \pi(i)} -\left(1-\frac{1}{k}\right)^{{\lfloor \frac{n}{2} \rfloor}-1}, 
  \end{align}
  and 
\begin{align}\label{Thm3_Upper}
    v_{n}(q) - v_{\infty}(\pi) \leq \frac{\W_{n-1} (M)}{\min(M^{k^2})} +   \left(1-\frac{1}{k}\right)^{n-2} .
\end{align}
\end{theorem}

Combining Theorem \ref{Thm3} with relations \eqref{Ergodic Convergence Thm} and \eqref{Bound on W_n}, we deduce the following.

\begin{corollary}\label{Coro3}
    The sequence of functions $v_{n}(\cdot) :\dk \to [0,1]$, $n\geq 1$, converges uniformly to $v_{\infty}(\pi)$ at a geometric rate. 
\end{corollary}

Corollary \ref{Coro3} can be viewed as a generalization of Theorem \ref{Thm1} for ergodic Markov chains. It implies that for such chains, for any prior distribution, the value of additional information perishes at a geometric speed. Moreover,  the converges of the values to $v_{\infty}(\pi)$ implies that as $n$ increases, the dependence of $v_n$ on the prior decreases exponentially fast.

\section{A Markov Decision Problem Approach}\label{Sec:MDP}

\subsection{Formal Definition and Motivation}\label{Subsec. MDP definition}

The fact that our model admits can be described by a Markov decision problem over the belief space $\dk$ follows from known arguments. However, for the sake of the proceeding analysis, we shall introduce and study the following specialized Markov decision problem $\G$:

\begin{itemize}
    \item The set of \textit{states} is $S = \dk \cup \mathscr{C}$, where $\C = \{m_1^*,...,m_k^*\}$ consists of a disjoint copy of the rows $m_1,...,m_k$ of the stochastic matrix $M$. 
    \item The set of \textit{actions} is the same as in the original model, i.e., $A =K$.  
    \item The \textit{transition rule} $\rho: S\times A \to \Delta(S)$ is defined by:
\begin{align*}
\rho(q,i) = \left\{
       \begin{array}{ll}
        m_i^* , & \hbox{with prob.\,\,}\,\,  q(i),	\,\,\,\, i \in K,\\ 
        w(q,i) M , & \hbox{with prob.\,\,}\,\, 1-q(i),	\,\,\,\, i \in K,
       \end{array}
     \right.
\end{align*}
for any $q \in \dk$ and $i \in A$, where $w(q,i) := (q - q(i)\delta_{\{i\}})/(1-q(i))$ describes the Bayesian update of $q$ given that the action $i \in K$ produced a negative signal, i.e., did not match the realized state. For any $m_i^* \in \C$ we set $\rho(m_i^*,\cdot) := \rho(m_i,\cdot)$.
\end{itemize}

A \textit{play} in $\G$ is an infinite sequence $(q_1,\xi_1,...,q_j,\xi_j,...)\in (S \times A)^{\N}$ of states and actions, where $q_1 =q$ is the initial prior. We assume that $q\in \dk$. For each $j\geq 1$, the $j$'th state of the MDP $q_j$ serves to describe the decision maker's conditional belief on $X_j$ given the past actions $\xi_1,...,\xi_{j-1}$ and the signals $I(\xi_1),..., I(\xi_{j-1})$ they produced. 

At any period $j\geq 1$, if the $j$'th action $\xi_j$ produced a negative signal (an event occurring with probability $1-q_{j}(\xi_j)$), the decision maker's posterior belief regarding the distribution of $X_j$ will move from $q_j$ to $w(q_j,\xi_j)$, and thus $w(q_j,\xi_j)M$ will describe the conditional distribution of $X_{j+1}$. Therefore, in such a case, the transition rule of $\G$ defines the next state, $q_{j+1}$, to equal to $w(q_j,\xi_j)M$. 

On the contrary, if the action $\xi_j$ produced a positive signal, i.e., matched the realized state (an event having probability $q_j(\xi_j)$), the conditional distribution of $X_{j+1}$ equals $\delta_{\xi_j}M = m_{\xi_j}$. However, in $\G$ we set $q_{j+1}$ to equal $m_{\xi_j}^*$, rather then $m_{\xi_j}$. The reason for this lies in the fact that we wish to make the pure strategy space in $\G$ equivalent to that in the original model. To do so, the decision maker has to observe the signals produced by her actions. If for instance, there exists $q \in\dk$ and action $i\in K$ such that 
\begin{align*}
    w(q,i)M = m_i,
\end{align*}
then regardless of the signal action $i$ produced, starting from the prior $q$, the conditional distribution of the next step of the Markov chain will equal $m_i$. Therefore, the set of states $\C$ was introduced to serve as a device, which within the MDP framework, allows the decision maker to learn the signals her actions produce. A central property of the MDP $\G$ is that starting from the second period on, a visit in state $m_i^* \in \C$ can occur if and only if the action $i \in K$ at the previous period produced a positive signal. 

\subsection{Strategies and Payoffs}

Let $H = (S \times A)^{\N}$ denote the space of all plays in the MDP $\G$. For every $j \geq 1$ let $H_j := (S \times A)^j$ denote the set of all histories of length $n$ in $\G$. A \textit{behavioral strategy} in $\G$ is a sequence of stage strategies $\s=(\s_j)_{j \geq 1}$ such that $\s_j : (S \times A)^{j-1}\times S \to \Delta(A)$ for every $j \geq 1$. By following $\s = (\s_j)_{j \geq 1}$, the decision maker's $j$'th stage action  $\xi_j \in A$ is chosen at random according to the lottery $\s_j(q_1,\xi_1,...,q_{j-1},\xi_{j-1},q_j) \in \Delta(A)$. Let $\S$ denote the set of all behavioral strategies in $\G$. 

The set of \textit{pure strategies} $\Sigma_{\mathcal{P}}\subset \S$ consists of all $\s = (\s_j)_{j \geq 1} \in \Sigma$ such that $\s_j : (S \times A)^{j-1}\times S \to A$ for every $j \geq 1$. That is, $\xi_j$ is chosen deterministically by $\s_j$ given $q_1,\xi_1,...,q_{j-1},\xi_{j-1},q_j$. The set $\S_{\mathcal{P}}$ is equivalent to the set of strategies $\Xi$ introduced in Section \ref{Sec_Intro}. The indicators sequence $I(\xi_1),I(\xi_2),...$, corresponds to the sequence of random variables $\textbf{1}\{q_2 \in \C\},\textbf{1}\{q_3 \in \C\},...$.

By Kolmogorov's Extension Theorem each behavioral strategy $\s \in \S$ together with the initial state $q \in \dk$, induce a unique probability measure $P_{q,\s}$ on the measurable space $(H,\H)$, where $\H$ is the $\sigma$-algebra generated by cylinder subsets of $H$, i.e., the set of finite histories $\cup_{j \geq 1} H_j$. Denote the expectation operator w.r.t.\ $P_{q,\s}$ by $E_{q,\s}$.

For each $n\geq 1$, consider the payoff
\begin{align*}
    u_n(q,\s):= E_{q,\s}\,  u(q_n,\xi_n), \quad \forall q\in \dk, \s \in \S,
\end{align*}
where $u(\cdot,\cdot): S \times A \to [0,1]$ is the linear extension of the utility $u$, i.e., 
\begin{align*}
    u(p,i)  := \sum_{i \in K}  p(i) \cdot u(i,\l), \quad \forall (p,\l) \in \dk \times A,
\end{align*}
   and 
   
   \begin{align*}   
   u(m_i^*,\l) := u(m_i,\l), \quad \forall (m_i^*,\l) \in \C \times A. 
\end{align*}
Define for every $q\in \dk$ and $n\geq 1$, the $n$'th value of the MDP $\G$ starting from $q$ by $\V_n (q):= \sup_{ \s \in \S} u_n(q,\s)$. By the dynamic programming principle,
\begin{align}\label{Eq_Bellman}
        \V_n(q)  & = \max_{i\in K} \lbrace q(i)\cdot \V_{n-1}(m_i^*)+(1-q(i))\cdot  \V_{n-1}(w(q,i)M) \rbrace.
\end{align}
By standard MDP theory (e.g., Chapter 1 in \cite{EilonBook}), such relation implies the existence of an optimal pure strategy, i.e., $\V_n (q):= \sup_{ \s \in \S_{\mathcal{P}}} u_n(q,\s)$. Thus, by the equivalence between $\S_{\mathcal{P}}$  and $\Xi$, $\V_n (q) = v_n(q)$ for every $n\geq 1$, and $q\in \dk$. Subsequently, we identify the values $\{v_n\}_{n\geq 1}$ with the values of the MDP $\Gamma$. 

Lastly, the proceeding analysis will take into account a general form of the dynamic programming principle presented in relation \eqref{Eq_Bellman}. Such form argues that for every  $\s \in \S$ satisfying $v_n(q) = u_n(q,\s)$,
\begin{align*}
    v_n(q) = E_{q,\s} \left( v_{n-t}(q_{t+1})\right),\quad  \forall n\geq 1, ~ \forall t=1,...,n-1.
\end{align*}

\section{Proofs}\label{Sec:Proofs}

\subsection{Preliminary Properties}

We devote the current discussion to several useful preliminary facts, which are commonly used throughout the proofs of the main results. 

\begin{lemma}\label{Convexity Lemma}
    For every $n\geq 1$, the function $v_n(\cdot):\dk \to [0,1]$ is convex.
\end{lemma}

\begin{lemma}
    \label{Beliefs Exp. Lemma}
    Let $t\geq 1$ be a positive integer, and let $D \in \H$ be an event measurable with respect to the sub $\s$-field generated by $(q_1,\xi_1,...,q_t,\xi_t)$ on $(H,\H)$. Then, for every $q\in \dk$ and $\s \in \S$ such that $P_{q,\s}(D)>0$,
    \begin{align*}
        E_{q,\s}\, (q_j \mid D) = E_{q,\s}\, (q_t \mid D)M^{j-t}, \quad \forall j\geq t.
    \end{align*}
\end{lemma}

The proofs of lemmas \ref{Convexity Lemma} and \ref{Beliefs Exp. Lemma}, which use standard arguments in the field of dynamic repeated games with incomplete information (e.g., Renault \cite{Renault_Survey}), can be found in the Appendix.

By setting $E=H$ and $t=1$ in Lemma \ref{Beliefs Exp. Lemma} we obtain the following useful corollary. 
\begin{corollary}\label{Mean-Consistency Property}
        For every $q\in \dk$ and $j \geq 1$,
    \begin{align*}
        E_{q,\s}\, q_j = qM^{j-1}, \quad \forall \s \in \S.
    \end{align*}
    In particular, if $q=\pi$, where $\pi$ is an invariant distribution of $M$ we have
    \begin{align*}
        E_{\pi,\s}\, q_j = \pi, \quad \forall \s \in \S, ~ \forall j\geq 1.
    \end{align*}
\end{corollary}

We conclude our preliminary properties with the following useful lemma. Intuitively, if the decision maker ignores the signals in the first $j$ stages, then her optimal payoff is $v_{n-j}(qM^j)$. Hence,

\begin{lemma}
    \label{Value along Trajectory Lemma}
    For every $q\in \dk$ and $n\geq 1$,
    \begin{align*}
        v_n(q) \geq v_{n-j}(qM^j), \quad \forall j=1,...,n-1.
    \end{align*}
\end{lemma}
\begin{proof}
    [Proof of Lemma \ref{Value along Trajectory Lemma}.]
    Fix $n\geq 1$ and $q\in \dk$. Let $\s \in \S$ be such that $v_n(q) = u_n(q,\s)$. By the dynamic programming principle and the optimality of $\s$, for any $j=1,...,n-1$:
    \begin{align*}
        u_n(q,\s) = E_{q,\s}\, v_{n-j} (q_{j+1}) \geq v_{n-j} \left( E_{q,\s}\, q_{j+1} \right) = v_{n-j} \left(qM^j \right),
    \end{align*}
    where the inequality follows from Lemma \ref{Convexity Lemma}
 and Jensen's inequality, and the second equality follows from Lemma \ref{Mean-Consistency Property}.  \end{proof}

Following the introduction and survey of the relevant preliminaries, we are in position to prove the results stated in Section \ref{Sec:Main_Results}. The initial step is to establish Proposition \ref{Prop3}
which has a central role in the proofs of Theorems \ref{Thm1} and \ref{Thm3}.

\subsection{Proof of Proposition \ref{Prop3}}

Fix $n\geq 1$. Let $\l \in K$ be the optimal action at the $n$-stage game starting from $q = m_i$. By the dynamic programming principle,
\begin{align}\label{Prop3_eq1}
    v_n (m_i) = m_i (\l)~ v_{n-1} (m^*_{\l}) + (1-m_i(\l))v_{n-1}(w(m_i,\l)M),
\end{align}
where we recall that $w(m_i,\l)$ was defined in Subsection \ref{Subsec. MDP definition} and equals to the Bayesian update of $m_i$ given that the signal produced by action $\l$ was negative. Since
\begin{align*}
    w(m_i,\l)M = \sum _{j \neq \l} \frac{m_i (j)}{1-m_i(\l)}\cdot m_j,
\end{align*}
we may employ the convexity of $v_{n-1}$, relation \eqref{Prop3_eq1}, and the fact that $v_{n-1} (m^*_{\l}) = v_{n-1} (m_{\l})$ to obtain
\begin{align}\label{Prop3_eq2}
    v_n (m_i) & \leq m_i (\l)~ v_{n-1} (m^*_{\l}) +  \sum _{j \neq \l} \frac{m_i (j)}{1-m_i(\l)}\cdot v_{n-1}( m_j) \nonumber \\
  &  = \sum_{j \in K} m_i (j) v_{n-1}(m_j).
\end{align}

As $i$ and $n$ were arbitrary, relation \eqref{Prop3_eq2} holds for every $i\in K$ and $n\geq 1$. By taking the weighted average with respect to $\pi$ we obtain:
\begin{align*}
   \Phi_n (\pi) = \sum_{i \in K} \pi(i)~ v_n (m_i) & \leq \sum_{i\in K} \pi(i)\sum_{j \in K} m_i (j) v_{n-1} (m_{j})  \\ 
    & = \sum_{j\in K} \left( \sum_{i\in K} \pi(i)~ m_i (j)\right)v_{n-1} (m_{j})\\
    & = \sum_{j \in K} \pi(j)~ v_{n-1} (m_{j}) = \Phi_{n-1} (\pi),
\end{align*}
where the second inequality holds since $\pi M = \pi$. Therefore, the sequence $\Phi_n (\pi) := \sum_{i \in K} \pi(i)~ v_n (m_i)$ is non-increasing in $n$, as required. \qed

\subsection{Proof of Theorem \ref{Thm1}}

 Corollary \ref{Mean-Consistency Property} and Lemma \ref{Value along Trajectory Lemma} imply that $\{v_n(\pi):n\geq 1\}$ is non-decreasing. By the convexity of $v_n$ and the definition of $\Phi_n (\pi)$ we have $v_n(\pi) \leq \Phi_n (\pi) $ for every $n\geq 1$. Therefore, as $\{v_n(\pi):n\geq 1\}$ is non-decreasing, and by Proposition \ref{Prop3} $\{\Phi_n (\pi): n\geq 1\}$ is non-increasing,
\begin{align}\label{Thm1_eq1}
    0\leq v_{\infty}(\pi) - v_n(\pi) \leq \Phi_n (\pi)-v_n(\pi), \quad \forall n\geq 1,
\end{align}
where $v_{\infty}(\pi)$ is the limit of the non-decreasing sequence $\{v_n(\pi):n\geq 1\}$. By relation \eqref{Thm1_eq1},  to establish inequality \eqref{Eq. Thm1} it is sufficient to show that
\begin{align}\label{Thm1_eq2}
    \Phi_n (\pi)-v_n(\pi) \leq \left(1- \frac{1}{k}\right)^{n-1}, \quad \forall n\geq 1.
\end{align}
To show the above inequality let us fix $n\geq 1$ and consider the strategy $\s^* \in \S$ which plays in the $n$'th stage game as follows:
\begin{itemize}
    \item Starting from stage $t=1$, select at each stage an action from $\{1,...,k\}$ uniformly at random, until the first time a positive signal occurs.
    \item Once a positive signal occurs, say at stage $j$, play optimally from stage $j+1$ to stage $n$.
\end{itemize}

Define
\begin{align*}
    & A_2 := \{q_2 \in \C\}\\
    & A_j := B_{j-1} \cap \{q_j \in \C\}, ~~\forall j=3,....,n,  
\end{align*}
where $B_{j} := \{ q_2\notin \C, ..., q_{j}\notin \C \}$ for all $j=2,...,n$.

For each $j=2,...,n$, the event $B_j$ describes the case in which the decision maker actions in stages $1,...,j-1$ did not produce a positive signal. The event $A_j$ describes the case where the first positive signal was produced by the $j-1$'st action of the decision maker.

Under a uniformly random action, for any prior $q\in \dk$, the probability to match the state, and thus produce a positive signal, is given by $\sum_{\l \in K} q(\l) \cdot (1/k) = (1/k) \sum_{\l \in K} q(\l) = 1/k$. Together with the definition of $\s^*$, this implies that
\begin{align}\label{Thm1_eq3}
    P_{\pi,\s^*}(A_2) = \frac{1}{k} \quad \text{and} \quad P_{\pi,\s^*}(q_j \in \C\mid B_{j-1}) = \frac{1}{k}, \quad  \forall~  j=3,...,n.
\end{align}

Moreover, since for every $2\leq j\leq n$, $H = A_2\dot{\cup} A_3\dot{\cup} ... \dot{\cup} A_j \dot{\cup} B_j$, where $H$ is the space of all plays in the MDP $\G$, by induction on $j$, relation \eqref{Thm1_eq3} implies  that  
\begin{align}\label{Thm1_eq4}
P_{\pi,\s^*} (A_j) = \left(1-\frac{1}{k}\right)^{j-2} \frac{1}{k},\quad \forall 2\leq j\leq n,
\end{align}
and
\begin{align}\label{Thm1_eq5}
    P_{\pi,\s^*} (B_{j}) = \left(1-\frac{1}{k}\right)^{j-1}, \quad \forall 2\leq j\leq n.
    \end{align}

The next result states that under $\s^*$, for each $j=2,...,n$, the probability that the first positive signal was produced by action $i$ at stage $j-1$ equals $\pi(i)$.

\begin{lemma}\label{Lemma Thm1}
    For every $2\leq j\leq n$ and every $i\in K$, 
\begin{align*}
    P_{\pi,\s^*} (q_j = m^*_i \mid A_j) = \pi(i).
\end{align*}
\end{lemma}

\begin{proof}[Proof of Lemma \ref{Lemma Thm1}.]
Fix $i\in K$. The proof proceeds by induction on $j$. The basis of the induction, i.e., $j=2$, follows from the fact that by the definition of $\s^*$ we have:
\begin{align*}
  P_{\pi,\s^*} (q_2 = m^*_i \mid A_2) & = \frac{P_{\pi,\s^*} (q_2 = m^*_i)}{P_{\pi,\s^*} (A_2)}\\
& = \frac{(1/k)\cdot \pi(i)}{1/k} = \pi(i),
\end{align*}
where the first equality follows from the fact that  $\{q_2 = m_i^*\} \subset A_2$, and the  second equality follows from relation \eqref{Thm1_eq3}.

To show the induction step $j-1 \Rightarrow{} j$, observe that since $H = A_2\dot{\cup} A_3\dot{\cup} ... \dot{\cup} A_{j-1} \dot{\cup} B_{j-1}$, 
\begin{align}\label{Thm1_Eq6}
    E_{\pi, \s^*} (q_{j-1}\mid B_{j-1}) = \frac{E_{\pi, \s^*} (q_{j-1}) - \sum_{t=2}^{j-1} P_{\pi,\s^*}(A_t)~ E_{\pi, \s^*} (q_{j-1}\mid A_{t}) }{P(B_{j-1})} 
\end{align}
 By the induction step, for every $t=2,...,j-1$, we have $P_{\pi,\s^*} (q_t = m^*_i \mid A_t) = \pi(i)$, so that 
 \begin{align}\label{Eq.*}
 E_{\pi, \s^*} (q_t\mid A_{t})=\pi, \quad  \forall 2\leq t< j.
 \end{align} Since by identity \eqref{Thm1_eq4} $P_{\pi,\s^*}(A_t) >0$ for every $t=2,...,n$, and since $A_t$ is measurable with respect to the sub $\s$-field generated by $(q_1,...,q_t)$ on $(H,\H)$, relation \eqref{Eq.*} together with Lemma \ref{Beliefs Exp. Lemma} imply that $E_{\pi, \s^*} (q_{j-1}\mid A_{t}) = \pi$ for every $2\leq t< j$.
 Substituting the latter relation back into  Eq. \eqref{Thm1_Eq6}, and using the fact that $E_{\pi, \s^*} \, q_{j-1} =\pi$ by Corollary \ref{Mean-Consistency Property}, we obtain that
\begin{align}\label{Thm1_Eq_Exp_B_j}
    E_{\pi, \s^*} (q_{j-1}\mid B_{j-1}) = \frac{\pi \left(1- \sum_{t=2}^{j-1} P_{\pi,\s^*}(A_t)\right) }{P(B_{j-1})} = \pi.
\end{align}

Put differently, under $\s^*$, as long as a positive signal was yet to occur, starting from the prior $\pi$, the expectation of the current state of $\G$ remains $\pi$.

The definitions of $\s^*$ and $A_j$ imply that for every $j\geq 3$,
\begin{align}\label{Thm1_eq7}
    P_{\pi,\s^*} (q_j = m^*_i \mid A_j) & = \frac{P_{\pi,\s^*} (\{q_j = m^*_i\} \cap A_j)}{P_{\pi,\s^*} (A_j)} \nonumber \\
    & = \frac{P_{\pi,\s^*} (\{q_j = m^*_i\} \cap B_{j-1})}{P_{\pi,\s^*} (A_j)}  \\
  & = \frac{P_{\pi,\s^*} (q_j = m^*_i \mid B_{j-1}) \cdot P_{\pi,\s^*}(B_{j-1})}{P_{\pi,\s^*} (A_j)}, \nonumber 
\end{align}
where the second equality follows from the fact that the event $\{q_j = m^*_i\} \cap B_{j-1}$ describes the case where the first positive signal was produced by action $i\in K$ at stage $j-1$, and thus $\{q_j = m^*_i\} \cap B_{j-1} = \{q_j = m^*_i\} \cap A_j$. 

As by the definition of $\s^*$, given $B_{j-1}$ the action at stage $j-1$ is uniform and independent of all past events,  
\begin{align*}
    P_{\pi,\s^*} (q_j = m^*_i \mid B_{j-1}) & =   E_{\pi,\s^*}\left(\frac{1}{k} \cdot q_{j-1} (i) \mid B_{j-1}\right)\\
    & = \frac{1}{k} \cdot \left[ \left( E_{\pi,\s^*}(q_{j-1} \mid B_{j-1})\right) (i)\right] = \frac{1}{k}\cdot  \pi(i),
\end{align*}
where in the second inequality we used the affinity of the function $q \mapsto q(i)$, together with linearity of the expectation operator, and in the third equality relation \eqref{Thm1_Eq_Exp_B_j}. Plugging this relation back into Eq. \eqref{Thm1_eq7} we obtain by identities \eqref{Thm1_eq4} and \eqref{Thm1_eq5} that
\begin{align*}
    P_{\pi,\s^*} (q_j = m^*_i \mid A_j) & = \frac{(1/k)\cdot  \pi(i) \cdot P_{\pi,\s^*}(B_{j-1})}{P_{\pi,\s^*} (A_j)}\\
    & = \frac{(1/k)\cdot  \pi(i) \cdot \left(1-\frac{1}{k}\right)^{j-2}}{\left(1-\frac{1}{k}\right)^{j-2} \frac{1}{k}}\\
    & = \pi(i).
\end{align*}
This completes the induction step, and thus the proof of the lemma.
\end{proof}

We proceed with the deduction of inequality \eqref{Thm1_eq2}. We have
\begin{align*}
   u_n(\pi, \s^*) & = E_{\pi, \s^*}\, u(q_n,\xi_n)\\
   & \geq  \sum_{j=2}^{n} P_{\pi,\s^*}(A_j)~ E_{\pi, \s^*} (u(q_n,\xi_n) \mid A_{j})\\ 
    &  = \sum_{j=2}^{n} P_{\pi,\s^*}(A_j)~ \sum_{i \in K} \pi (i
    ) \cdot v_{n+1-j} (m^*_i)\\
    & = \sum_{j=2}^{n} P_{\pi,\s^*}(A_j)\cdot \Phi_{n+1-j}(\pi) \geq  \sum_{j=2}^{n} P_{\pi,\s^*}(A_j)\cdot \Phi_{n}(\pi) \\
   & = P_{\pi,\s^*}\left(\,\bigcup_{j=2}^n A_j\right)\cdot \Phi_{n}(\pi) = \left( 1-P_{\pi,\s^*}(B_n)\right)\cdot \Phi_n (\pi)\\
   & = \Phi_n (\pi) - \left(1 - \frac{1}{k}\right)^{n-1}\Phi_n (\pi) \geq \Phi_n (\pi) - \left(1 - \frac{1}{k}\right)^{n-1}
\end{align*}
where the first inequality follows from the non-negativity of $u$; the second equality from the definition of $\s^*$ and Lemma \eqref{Lemma Thm1}; the third equality from the definition of $\{\Phi_j(\pi):j\geq 1\}$; the second inequality from Proposition \ref{Prop3}, the fourth and fifth equalities from the fact that $H = A_2\dot{\cup} A_3\dot{\cup} ... \dot{\cup} A_{n} \dot{\cup} B_{n}$, the sixth equality from identity \eqref{Thm1_eq5}, and the last inequality from $\Phi_n (\pi) \leq 1$.

The above derivation implies that 
\begin{align*}
    \Phi_n (\pi) - u_n(\pi, \s^*) \leq \left(1-\frac{1}{k}\right)^{n-1}. 
\end{align*}
Since $u_n(\pi,\s^*) \leq v_n(\pi)$, we conclude
\begin{align*}
    \Phi_n (\pi) -v_n(\pi) & \leq \Phi_n (\pi) -u_n(\pi,\s^*)\leq  \left(1-\frac{1}{k}\right)^{n-1}.
\end{align*}
As $n\geq 1$ was arbitrary, this establishes the sufficient condition in inequality \eqref{Thm1_eq2}, thus completing the proof of Theorem \ref{Thm1}.
\qed

\subsection{Proof of Theorem \ref{Thm2}}

Let $(Y_j)_{j\geq 1}$ be the $\{0,1\}$-valued process that indicates positive signals. It is defined by $Y_1:=0$ and $Y_j:= \textbf{1}\{q_j \in \C\}$ for every $j\geq 2$. Let $\F = (\F_j)_{j\geq 1}$ be the natural filtration with respect to the process $(Y_j)_{j\geq 1}$.

Let $\S^* \subset \S_{\mathcal{P}}$ denote the space of strategies $\s = (\s_j)_{j\geq 1}$ such that $\s_j(q_1,\xi_1,...,q_{j-1},\xi_{j-1},q_j)$ is $\F_j$-measurable. That is, a strategy $\s \in \S_*$ can use the information of which stages had positive signals, but not the identity of the states of the Markov chain at those stages. Roughly speaking, $\S^*$ describes all possible binary decision trees, parsed according to the signals produced by past actions. Let us note that any strategy $\xi \in \Xi$ in the original model induces a strategy $\s \in \S^*$. Indeed, for any $j\geq 1$ and binary sequence $x \in \{0,1\}^j$ with $x_1 = 0$ (corresponding to $Y_1 :=0$) define
\begin{align*}
    & \s_1(x_1)  := \xi_1, \\
    & \s_j (x_1,...,x_{j}) := \xi_j(\xi_1,x_2, \xi_2,x_3,..., \xi_{j-1},x_{j}),
\end{align*}
where we recall that $\xi_1$ is deterministic, $\xi_2$ is deterministic given $\xi_1,x_2$, $\xi_3$ is deterministic given $\xi_1,x_2,\xi_2,x_3$, etc. Therefore, by the equivalence of the MDP $\G$ to the original model it follows that for every prior $q \in \dk$ and $n\geq 1$ there must exist $\s^*_{q,n} \in \S^*$ such that $u_n(q,\s^*_{q,n}) = v_n(q)$.

The main advantage in considering such binary decision trees is that the payoff function of any strategy in $\S^*$ is linear in prior. The proof  is deferred to the Appendix.

\begin{lemma}\label{DT_Lemma}
    For every $\a \in (0,1)$, $p,q \in \dk$, and $n\geq 1$,
    \begin{align*}
        u_n(\a q +(1-\a)p,\s) = \a u_n(q,\s) + (1-\a) u_n(p,\s), \quad \forall \s \in \S^*. 
    \end{align*}
\end{lemma}

To proceed with the proof of the theorem, let us first observe that a Markov chain starting from an invariant distribution $\pi$ can only visit states that are in the support of $\pi$. Therefore, we may assume without loss of generality that $\pi$ has full support, i.e., $\pi(i)>0$ for every $i \in K$. Also, let us recall that for every $n\geq 1$, $\s^*_{\pi,n} \in \S^*$ is the binary decision tree satisfying $v_n(\pi) = u_n(\pi, \s^*_{\pi,n})$.

\begin{claim}\label{Claim_Vertex_Opt}\label{Prop1_Claim1}
    Assume that $v_{n+1}(\pi) = v_n(\pi)$. Then, $v_n(m_i) = u_n(m_i, \s^*_{\pi,n})$ for every $i\in K$.
\end{claim}

\begin{proof}[Proof of Claim \ref{Claim_Vertex_Opt}.]
    Consider the following strategy $\s \in \S$: at the first stage of $\G$, pick an action at uniform from $K$; in the next $n$ stages follow the strategy $\s^*_{\pi,n}$. On the one hand, we have by the dynamic programming principle, Lemma \ref{DT_Lemma}, and Corollary \ref{Mean-Consistency Property}, 
    \begin{align*}
        u_{n+1}(\pi, \s ) & = E_{\pi,\s } u_n(q_2,\s^*_{\pi,n}) = u_n(E_{\pi,\s }\, q_2,\s^*_{\pi,n})\\
        &= u_{n}(\pi, \s^*_{\pi,n}) = v_n(\pi) = v_{n+1}(\pi),
    \end{align*}
    implying that $\s$ is optimal in the $(n+1)$-stage problem starting from $\pi$. On the other hand, by the definition of $\s$ we have
    \begin{align*}
     u_{n+1}(\pi, \s ) & = \frac{1}{k}   \sum_{i\in K} \pi(i)\cdot u_n(m_i^*,\s^*_{\pi,n})\\
     & \quad + \left(1-\frac{1}{k}\right)\sum_{i\in K} (1-\pi(i))\cdot u_n(w(\pi,i)M,\s^*_{\pi,n}).
    \end{align*}
    Therefore, by the optimality of $\s$, and the fact that $\pi(i)>0$ for every $i\in K$, \begin{align*}
    u_n(m_i,\s^*_{\pi,n}) = u_n(m_i^*,\s^*_{\pi,n})= v_n(m^*_i)=v_n(m_i), \quad \forall i \in K,
    \end{align*}
    as desired. 
\end{proof}

Let $M(\dk) := \{qM: q \in \dk\}$. We obtain that $\s^*_{\pi,n}$ is optimal for all $q \in M(\dk)$:
  
  \begin{corollary}\label{Prop1_Cor1}
      Assume that $v_{n+1}(\pi) = v_n(\pi)$. Then, for every $q \in M(\dk)$ it holds that $v_n(q) = u_n(q, \s^*_{\pi,n})$.
  \end{corollary}

  \begin{proof}[Proof of Corollary \ref{Prop1_Cor1}.]
      Since $m_1,...,m_k$ are the extreme points of $M(\dk)$, there exist convex weights $(\b_i)$ such that $q = \sum_{i=1}^k \b_i m_i$. Therefore, using Lemma \ref{DT_Lemma} and Claim \ref{Prop1_Claim1} we obtain
      \begin{align*}
          u_n(q,\s^*_{\pi,n}) = \sum_{i\in K} \b_i \cdot u_n(m_i,\s^*_{\pi,n}) = \sum_{i\in K} \b_i \cdot v_n(m_i) \geq v_n(q), 
      \end{align*}
      where the inequality follows from the convexity of $v_n$. As $v_n(q) \geq u_n(q,\s^*_{\pi,n})$ by definition, it follows that $v_n(q) = u_n(q,\s^*_{\pi,n})$ as desired.
  \end{proof}

  We are now in position to establish Theorem \ref{Thm2}. Namely, we need to show that if $v_{n+1}(\pi) = v_n(\pi)$ then $v_N(\pi) = v_n(\pi)$ for every $N\geq n+1$. Let us fix such $N$, and consider the following strategy $\s$ for the $N$-stage problem:
  \begin{itemize}
      \item At the first $N-n$ stages follow the strategy $\s^*_{N,\pi}$.
      \item Starting from stage $N-n+1$, forget all past information, and start following $\s^*_{n,\pi}$ until stage $N$.
  \end{itemize}

   Put differently, the strategy $\s$ follows the optimal decision tree for the $N$-stage problem in the first $N-n$ stages, and then for the last $n$ stages, $\s$ switches to the decision tree $\s^*_{\pi,n}$. As by the rules of $\G$, $q_{N-n+1} \in M(\dk)$, the dynamic programming principle together with Corollary \ref{Prop1_Cor1} imply that $\s$ must be optimal in the $N$-stage problem. In particular, applying the dynamic programming principle with the optimality of $\s$ yields
   \begin{align*}
       v_N (\pi) & = E_{\pi,\s}\, v_n(q_{N-n+1}) = E_{\pi,\s}\, u_n(q_{N-n+1},\s^*_{n,\pi})\\
       & = u_n(E_{\pi,\s}\,q_{N-n+1},\s^*_{n,\pi}) = u_n(\pi,\s^*_{n,\pi}) = v_n(\pi),
   \end{align*}
  where the second equality follows from the definition of $\s$, the third from Lemma \ref{DT_Lemma}, the fourth from Corollary \ref{Mean-Consistency Property}, and the fifth from the optimality of $\s^*_{N,\pi}$ in the $n$'th stage problem. Therefore, we obtain that $v_N (\pi) = v_n(\pi)$ for every $N\geq n+1$ as required, thus completing the proof of the theorem. \qed

\subsection{Proof of Proposition \ref{Prop.1}}\label{Subsec_Proofs_Prop1}

For the matching utility $u_*$, for every $ q\in \dk$, $\s \in \S$, and $n\geq 1$,
\begin{align*}
    u_n(q,\s) & = E_{q,\s} \left( E_{q,\s}\, ( u_*(q_n,\xi_n) \mid q_1,\xi_1,...,q_{n-1},\xi_{n-1},q_n,\xi_n ) \right)\\
    & = E_{q,\s} \left( \sum_{i\in K} q_n(i) \cdot \xi_n(i)\right).
\end{align*}
Therefore, for every $n\geq 1$, to maximize her expected payoff in the $n$-stage problem, the decision maker's action at stage $n$ must be a state $i\in K$ which has maximal likelihood under $q_n$. Therefore, in the case of a matching utility,
\begin{align*}
    v_n (q) = \max_{\s \in \S} E_{q,\s} (\max\, q_n)
\end{align*}
for all $n\geq 1$.

Let us now present the details of the example. Let $K = \{1,2,3,4\}$ and define the matrix $M$ by 
\begin{align*}
    M = \begin{bmatrix}
0 & 0.5 & 0 & 0.5 \\
0.5 & 0 & 0.5 & 0 \\
0 & 0.5 & 0 & 0.5 \\
0.5 & 0 & 0.5 & 0
\end{bmatrix}.
\end{align*}
The unique invariant distribution of $M$ is $\pi = (1/4,1/4,1/4,1/4)$. 
The matrix $M$ has two useful properties:
\begin{itemize}
    \item[(i)] If $\#\text{support}(p)\geq 3 \Rightarrow \#supp(pM) = 4$. 
    \item[(ii)] If $\#\text{support}(p)=3$ then $\max pM<1/2$.
\end{itemize}

While property (i) follows directly, to showcase the validity of property (ii) we consider the example where $p = (\a,\b,1-\a-\b,0) \in \dk$ such that $0<\a,\b$ and $\a+\b<1$. In such a case, $pM = (\b/2, (1-\b)/2, \b/2, (1-\b)/2)$, and thus $\max pM<1/2$, as required. The other cases can be verified in a similar fashion.

\begin{lemma}\label{Prop.1_Lemma}
    For every $n\geq 1$ we have $v_n(\pi) < 1/2$.
\end{lemma}

\begin{proof}[Proof of Lemma \ref{Prop.1_Lemma}.]
Let us reconsider the events $B_{j} := \{ q_2\notin \C, ..., q_{j}\notin \C \}$, $j\geq 2$, which were previously studied in the proof of Theorem \ref{Thm1}. For each $j\geq 2$, $B_j$ describes the case in which the decision maker was yet to produce a positive signal within her first $j-1$ actions. For each $j\geq 1$ consider the random variable $p_j:=w(q_j,\xi_j)$. In words, $p_j$ describes the Bayesian update of $q_n$ after learning that the $j$'th action $\xi_j$ produced a negative signal.

By property (i) of $M$ and the fact that $\pi$ has full support we have that starting from $\pi$, for every strategy $\s \in \S_{\mathcal{P}}$, conditional on the event $B_j$, $p_1,...,p_{j}$ are all supported on exactly 3 elements. Therefore, since conditional on $B_j$, $q_j = p_{j-1} M$, property (ii) of $M$ implies that for every $\s \in \S_{\mathcal{P}}$ and $j\geq 2$, conditional on $B_j$, $\max q_j <1/2$. In particular,
\begin{align}\label{Prop.1_Eq1}
    E_{\pi,\s}(\max q_j\mid B_j)<\frac{1}{2}, \quad \forall j\geq 2, \forall \s \in \S_{\mathcal{P}}.
\end{align}

Let us fix $n\geq 2$. The definition of $M$ assures that $P_{\pi,\s}(B_n)>0$ for every strategy $\s \in \S_{\mathcal{P}}$ and every $n\geq 2$. Using the latter, together with inequality \eqref{Prop.1_Eq1}, and the fact that $\max q_n \leq 1/2$ by the definition of $M$ and, we obtain:
\begin{align*}
    E_{\pi,\s} (\max\, q_n) \leq \frac{1}{2} P_{\pi,\s}(B_n^c)+P_{\pi,\s} (B_n) E_{\pi,\s}(\max q_n\mid B_n)<\frac{1}{2}.
    \end{align*}
    
Since the strategy space $\S_{\mathcal{P}}$ for the $n$-stage problem contains only finitely many strategies, we deduce that $v_n(\pi)<1/2$ for every $n\geq 2$. Therefore, by Theorem \ref{Thm1} we also have that $v_n(\pi)<1/2$ for every $n\geq 1$, as desired. 
\end{proof}

Next, we claim the following.

\begin{claim}\label{Prop.1_Claim}
    For every row $m_i$ of $M$ and every $n\geq 1$, we have $v_n(m_i) = 1/2$.
\end{claim}

\begin{proof}[Proof of Claim \ref{Prop.1_Claim}.]
    The claim holds trivially for $n=1$ as $v_1 (m_i) = \max m_i = 1/2$. Let us fix $n\geq 2$ and a row $m_i$. As discussed in Subsection \ref{Subsec:Main_Index}, starting from $m_i$, by selecting at each step one of the possible two states, the decision maker obtain full information upon observing the signal produced by her action. Therefore, $q_n \in \{m_1,...,m_4\} \cup \C$, which in turn implies that $v_n(m_i) = 1/2$.
\end{proof}

We are now in position to deduce Proposition \ref{Prop.1}. First, by Claim \ref{Prop.1_Claim} have that $\Phi_n (\pi) := \sum_{i \in K}\pi(i)\cdot v_n(m_i)=    1/2$ for every $n\geq 1$, and thus by relation \eqref{Thm1_eq2}, which was shown to hold in the proof of Theorem \ref{Thm1},  $v_n(\pi) \uparrow 1/2$ as $n \to \infty$. As by Lemma \ref{Prop.1_Lemma}  $v_n(\pi) < 1/2$ for every $n\geq 1$, we deduce that $i(\pi,M) = + \infty$, as required. \qed

\subsection{Proof of Theorem \ref{Thm3}}

\subsubsection{Proof of the Upper Bound \eqref{Thm3_Upper}}

We begin with the proof of the upper bound in Theorem \ref{Thm3}. The key step required to show such bound is stated in the following proposition. 

    \begin{proposition}\label{Thm3_Prop}
        For every $n\geq 2$ we have
        \begin{align*}
            \max_{i\in K} v_n(m_i) - \Phi_n (\pi) \leq \frac{\W_{n} (M)}{\min(M^{k^2})}.
        \end{align*}
    \end{proposition}

Before turning to the proof of Proposition \ref{Thm3_Prop}, let us first show its sufficiency for the deduction of \eqref{Thm3_Upper}. To do so, note that by the convexity of $v_n$, for every $q\in \dk$ and $n\geq 2$:

\begin{align*}
    v_{n}(q) - v_{\infty}(\pi) &\leq \sum_{i\in K} q(i)\cdot v_{n}(\delta_{\{i\}}) - v_{\infty}(\pi)\\
    & \leq \left(\max_{i\in K} v_n(\delta_{\{i\}}) - \Phi_{n-1} (\pi) \right)  + \left( \Phi_{n-1} (\pi) - v_{\infty}(\pi) \right)\\
    & \leq \left(\max_{i\in K} v_{n-1}(m_i) - \Phi_{n-1} (\pi) \right)  + \left( \Phi_{n-1} (\pi) - v_{n-1}(\pi) \right)\\
    & \leq  \frac{\W_{n-1} (M)}{\min(M^{k^2})} +   \left(1-\frac{1}{k}\right)^{n-2},
\end{align*}
as desired, where in the third inequality we used the relations $v_n(\delta_{\{i\}}) = v_{n-1}(m_i)$ and $v_{\infty}(\pi) \geq v_{n-1}(\pi)$, and the last inequality follows from Proposition \ref{Thm3_Prop} and Theorem \ref{Thm1}.
\\

We now turn to the proof of Proposition \ref{Thm3_Prop}.

\begin{proof}[Proof of Proposition \ref{Thm3_Prop}.]
    Consider the $(k^2+1+n)$-stage problem, where $n\geq 2$ is fixed. We introduce the strategy $\s^{**}\in \S$ defined as follows:

    \begin{itemize}
        \item Play arbitrary in the first $k^2$ stages.
        \item In stage $k^2+1$ select the state $\l(1)\in K$, where $\l(1)$ is chosen to satisfy $$v_n (m_{\l(1)}) =\max_{i\in K} v_{n}(m_i).$$ If the signal produced by $\l(1)$ was positive play optimally in the remaining $n$ stages.
        \item Starting from stage $k^2+2$ up to stage $k^2+n$ follow the next scheme: If no positive signal occurred at stages $k^2+1,...,k^2+j-1$, where $j=2,...,n$, select at stage $k^2+1+j$ the state $\l(j)\in K$, where $\l(j)$ is chosen to satisfy $$v_{n+1-j} (m_{\l(j)}) =\max_{i\in K} v_{n+1-j}(m_i).$$
        If the action $\l(j)$ produced a positive signal, play optimally in the remaining $n+1-j$ stages ($j=2,...,n$).
        \item At stage $k^2+1+n$, if no positive signal was produced along stages $k^2+1,...,k^2+n$, select the state $1$.
        \end{itemize}

Figure \ref{fig:sigma-star-star} provides an illustration of the strategy $\s^{**}$. 
\\

\begin{figure}[htbp]
\centering
\begin{tikzpicture}[
  >=Stealth,
  font=\small,
  node distance=9mm and 13mm,
  every node/.style={align=center},
  phase/.style={rectangle, draw, thick, rounded corners=2pt, fill=orange!8,
                minimum height=13mm, text width=25mm, inner sep=2mm},
  guess/.style={circle, draw, thick, fill=blue!10, minimum size=10mm, inner sep=0pt},
  finalguess/.style={rectangle, draw, thick, rounded corners=2pt, fill=black!8,
                minimum height=10mm, minimum width=20mm, inner sep=1mm},
  outcome/.style={rectangle, draw, thick, rounded corners=2pt, fill=green!8,
                  text width=28mm, minimum height=11mm, inner sep=1.5mm},
  stagelbl/.style={font=\scriptsize},
  dotslbl/.style={font=\scriptsize, text=black!60},
  hitarr/.style={-Stealth, thick, green!40!black, dashed},
  missarr/.style={-Stealth, thick, black},
]
 
\node[phase] (p0) {Stages $1,\dots,k^2$\\[1pt]\textit{arbitrary play}};
\node[guess, right=of p0] (n1) {$\l(1)$};
\node[right=12mm of n1] (dots) {\Large $\cdots$};
\node[guess, right=12mm of dots] (nn) {$\l(n)$};
\node[finalguess, right=of nn] (fin) {select $1$};
 
\node[stagelbl, below=3mm of n1] {stage $k^2\!+\!1$};
\node[dotslbl, below=3mm of dots]{};
\node[stagelbl, below=3mm of nn] {stage $k^2\!+\!n$};
\node[stagelbl, below=3mm of fin] {stage $k^2\!+\!1\!+\!n$};
 
\node[outcome, above=12mm of n1] (o1) {Guarantee $\max_{i\in K} v_{n}(m_i)$\\ in the remaining $n$ stages};
\node[outcome, above=12mm of nn] (on) {Guarantee $\max_{i\in K} v_{1}(m_i)$\\ in the remaining stage};
 
\draw[missarr] (p0) -- (n1);
\draw[missarr] (n1) -- node[stagelbl, above] {negative} (dots);
\draw[missarr] (dots) -- node[stagelbl, above] {negative} (nn);
\draw[missarr] (nn) -- node[stagelbl, above] {negative} (fin);
 
\draw[hitarr] (n1) -- node[stagelbl, left=1pt] {positive} (o1);
\draw[hitarr] (nn) -- node[stagelbl, left=1pt] {positive} (on);
 
\end{tikzpicture}
\caption{An illustration of the strategy $\sigma^{**}$ in the $(k^2+1+n)$-stage problem.}
\label{fig:sigma-star-star}
\end{figure}
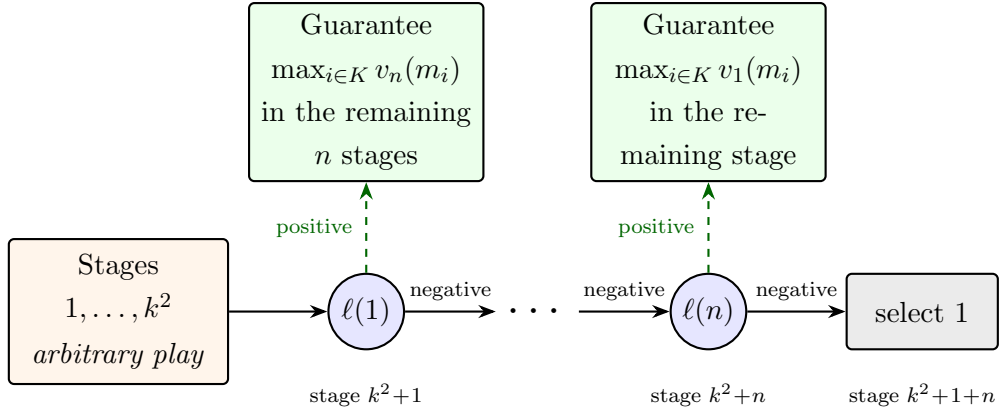

Define 
    \begin{align*}
        A'_1 & = \{q_{k^2+2} = m^*_{\l(1)}\},\\
        A'_j & = \{q_{k^2+1+j}  = m^*_{\l(j)}\}~\setminus~ \bigcup_{t=1}^{j-1} A'_t , \quad \forall  j=2,...,n.        
    \end{align*}
     and let $B' = \bigcap_{t=1}^n \overline{A'_t}$, where for each $t$, $\overline{A'_t}$ denotes the complement event of $A'_t$. For any $j=1,...,n$, the event $A'_j$ describes the case where first positive signal  along stages $k^2+1,...,k^2+n$, under the strategy $\s^{**}$, occurred at stage $k^2+j$. The event $B'$ describes the case where the $n$ actions along stages $k^2+1,...,k^2+n$, suggested by $\s^{**}$ produced negative signals.

      The role of the first $k^2$ stages is to ensure that for every prior $q\in \dk$ the expectation of $q_{k^2+1}$ under $\s^{**}$, which by Corollary \ref{Mean-Consistency Property} equals $qM^{k^2}$, puts a mass of at least $\min(M^{k^2})>0$ on every state $i \in K$. In that case, the action $\l(1)$ at stage $k^2+1$, which is independent of all past actions and their signals, has a probability of at least $\min (M^{k^2})$ of producing a positive signal, hence

\begin{align}\label{Eq.A_1'}
    P_{q,\s^{**}} (A_1') \geq \min(M^{k^2}), \quad \forall q \in \dk .
\end{align}

As the sequence $\l(1),...,\l(n)$ is a deterministic sequence, the definition of $\s^{**}$ and the Markov property of $\{X_t:t\geq 1\}$ imply that
\begin{align}\label{Eq.B'}
 P_{q,\s^{**}} (B')  & = P_{q,\s^{**}}\left(\{ q_{k^2+2}\neq m^*_{\l(1)},q_{k^2+3}\neq m^*_{\l(2)},..., q_{k^2+n+1}\neq m^*_{\l(n)}  \}\right) \nonumber \\
& = P_{qM^{k^2},\xi}\left(\{ X_{1}\neq \l(1),X_{2}\neq \l(2),..., X_{n}\neq \l(n)  \}\right)\\
& \leq \W_n(M), \quad \forall q \in \dk, \nonumber
    \end{align}
where $\xi \in \Xi$ is the strategy that selects deterministically $\l(1),...,\l(n)$ along stages $1,...,n$.

By the definition of $\s^{**}$ we have for every $q\in \dk$:
    \begin{align}\label{Thm3_Prop_eq1}
        u_{k^2+1+n} (q,\s^{**}) & = \sum_{j=1}^n P_{q,\s^{**}}(A'_j) \max_{i\in K} v_{n+1-j}(m_i) \nonumber\\
        & \quad + P_{q,\s^{**}}(B')E_{q,\s^{**}} \left( u(q_{k^2+1+n},1) \mid B' \right) \nonumber \\
        & = \Phi_n (\pi) + \sum_{j=1}^n P_{q,\s^{**}}(A'_j) \left(\max_{i\in K} v_{n+1-j}(m_i) - \Phi_n (\pi)\right)\\
        & \quad + P_{q,\s^{**}}(B') \left( E_{q,\s^{**}} \left( u(q_{k^2+1+n},1)\mid B' \right) - \Phi_n (\pi) \right). \nonumber
    \end{align}

\bigskip

Define for every $i\in K$:
        \begin{align}\label{Thm3:eqDelta_i}
            \Delta_i & := \sum_{j=1}^n P_{m_i,\s^{**}}(A'_j) \left(\max_{\l\in K} v_{n+1-j}(m_{\l}) - \Phi_n (\pi)\right) \nonumber\\
        & \quad + P_{m_i,\s^{**}}(B') \left( E_{m_i,\s^{**}} \left(u(q_{k^2+1+n},1)\mid B' \right) - \Phi_n (\pi) \right)\\
        & = u(m_i,\s^{**}) - \Phi_n (\pi). \nonumber
        \end{align}

\begin{claim}\label{Thm3_Claim}
    There exists $i\in K$ such that $\Delta_i \leq 0$.
\end{claim}

\begin{proof}[Proof of Claim \ref{Thm3_Claim}.]
    Assume in contradiction that the corollary of the claim does not hold, i.e., $\Delta_i>0$ for every $i\in K$. Then, using relations \eqref{Thm3_Prop_eq1} and \eqref{Thm3:eqDelta_i} we obtain
    \begin{align*}
            \Phi_{k^2+1+n}(\pi)= \sum_{i\in K} \pi(i)\cdot v_{k^2+1+n}(m_i) & \geq \sum_{i\in K} \pi(i)\cdot u_{k^2+1+n}(m_i,\s^{**})\\
            & = \sum_{i\in K} \pi(i) \left(\Phi_n (\pi) + \Delta_i\right)\\
            & > \sum_{i\in K} \pi(i)\cdot \Phi_n (\pi) = \Phi_n (\pi),
        \end{align*}
        a contradiction to the fact that $\{\Phi_n (\pi): n\geq 1\}$ is non-increasing.
\end{proof}

Building on Claim \ref{Thm3_Claim}, let $\l \in K$ be such that $\Delta_{\l}\leq 0$, i.e.,

\begin{align*}
     & \sum_{j=1}^n P_{m_{\l},\s^{**}}(A'_j) \left(\max_{i\in K} v_{n+1-j}(m_i) - \Phi_n (\pi)\right)\\
        & \quad + P_{m_{\l},\s^{**}}(B') \left( E_{m_{\l},\s^{**}} \left(u(q_{k^2+1+n},1)\mid B' \right) - \Phi_n (\pi) \right) \leq 0,
\end{align*}
and hence
\begin{align}\label{Thm3:eq_decomposition}
   & P_{m_{\l},\s^{**}}(A'_1) \left(\max_{i\in K} v_{n}(m_i) - \Phi_n (\pi)\right) + \sum_{j=2}^n P_{m_{\l},\s^{**}}(A'_j) \left(\max_{i\in K} v_{n+1-j}(m_i) - \Phi_n (\pi)\right) \nonumber \\
  & \quad + P_{m_{\l},\s^{**}}(B') E_{m_{\l},\s^{**}} \left(u(q_{k^2+1+n},1)\mid B' \right) \leq P_{m_{\l},\s^{**}}(B') \cdot \Phi_n (\pi).
\end{align}
Since $\max_{i\in K} v_{n+1-j}(m_i)\geq \Phi_{n+1-j}(\pi)\geq \Phi_n (\pi)$ for every $j=2,...,n$, and since $u$ is non-negative, relation \eqref{Thm3:eq_decomposition} implies:
        \begin{align}\label{Thm3_Prop_eq2}
 P_{m_{\l},\s^{**}}(A'_1) \left(\max_{i\in K} v_{n}(m_{i})-\Phi_n (\pi)\right)\leq P_{m_{\l},\s^{**}} (B')\cdot \Phi_n (\pi).
    \end{align}
Combining the inequalities \eqref{Eq.A_1'} and \eqref{Eq.B'} for the prior $q = m_{\l}$ with relation \eqref{Thm3_Prop_eq2} and the fact that $\Phi_n (\pi)\leq 1$, we obtain
    \begin{align*}
        \max_{i\in K} v_{n}(m_{i})-\Phi_n (\pi) \leq \frac{P_{m_{\l},\s^{**}} (B')}{P_{m_{\l},\s^{**}}(A'_1)}\leq \frac{\W_n(M)}{\min(M^{k^2})},
    \end{align*}
    thus establishing the proposition.
  \end{proof}

\subsubsection{Proof of the Lower Bound \eqref{Thm3_Lower}}

To prove the lower bound we first show the following convexity lemma.

    \begin{lemma}\label{Thm3_Convexity_Lemma}
        Assume $f:\dk \to [0,1]$ is a convex function. For every $p \in \emph{int}(\dk)$ and $q\in \dk$,
        \begin{align*}
            f(q)-f(p) \geq -\left(\frac{1}{\min_{i\in K} p(i)}\right) \Vert q-p\Vert_1.
        \end{align*}
    \end{lemma}

    \begin{proof}[Proof of Lemma \ref{Thm3_Convexity_Lemma}.]
        Fix $p \in \emph{int}(\dk)$ and $q\in \dk$, and let $z_q$ be the point on the intersection of the ray $\{\theta p + (1-\theta)q: \theta\geq 0\}$ with $\partial\dk$. By the convexity of $f$ we have 
        \begin{align}\label{Eq.Convexity.Lemma}
            f(p) \leq \frac{\Vert q-p \Vert_1 }{\Vert q-z_q \Vert_1} f(z_q) + \frac{\Vert p-z_q \Vert_1 }{\Vert q-z_q \Vert_1} f(q).
       \end{align}
       Using relation \eqref{Eq.Convexity.Lemma} together with the inequalities $\Vert q-z_q \Vert_1 \geq \Vert p-z_q \Vert_1 \geq \min_{i\in K} p(i)$, we obtain:
       \begin{align*}
           f(q) - f(p) & \geq -\frac{\Vert q-p \Vert_1 }{\Vert q-z_q \Vert_1}(f(z_q)-f(q)) \\
           & \geq -\frac{\Vert q-p \Vert_1}{\Vert q-z_q \Vert_1}f(z_q)\\
           & \geq -\frac{\Vert q-p \Vert_1}{\Vert p-z_q \Vert_1}\\
           & \geq -\left(\frac{1}{\min_{i\in K} p(i)}\right) \Vert q-p\Vert_1,
       \end{align*}
        where the second inequality holds since $f(q)\geq 0$, and the third since $f(z_q)\leq 1$. The claim follows. 
        \end{proof}

    We are now in position to prove the lower bound in Theorem \ref{Thm3}. Fix $n\geq 2$. By Lemma \ref{Value along Trajectory Lemma},
    \begin{align}\label{Eq.Lower1}
        v_n(q) - v_{\infty}(\pi) & \geq v_{\lfloor \frac{n}{2} \rfloor}(qM^{\lceil \frac{n}{2} \rceil})- v_{\infty}(\pi)\\
        & = \left( v_{\lfloor \frac{n}{2} \rfloor}(qM^{\lceil \frac{n}{2} \rceil}) - v_{\lfloor \frac{n}{2} \rfloor}(\pi)\right) + \left( v_{\lfloor \frac{n}{2} \rfloor}(\pi) - v_{\infty}(\pi)\right) \nonumber.
    \end{align}
    Applying Lemma \ref{Thm3_Convexity_Lemma} for $f=v_{\lfloor \frac{n}{2} \rfloor}$ and $p=\pi$ we obtain:\footnote{Note that $\pi \in \text{int}(\dk)$ since the Markov chain is assumed to be irreducible.}
    \begin{align}\label{Eq.Lower2}
        v_{\lfloor \frac{n}{2} \rfloor}(qM^{\lceil \frac{n}{2} \rceil}) - v_{\lfloor \frac{n}{2} \rfloor}(\pi) & \geq - \left(\frac{1}{\min_{i\in K} \pi(i)}\right) \Vert qM^{\lceil \frac{n}{2} \rceil}-\pi\Vert_1 \nonumber\\
        & \geq - \frac{d_{\lceil \frac{n}{2} \rceil} (M)}{\min_{i\in K} \pi(i)}.
    \end{align}
    By Theorem \ref{Thm1}, $v_{\lfloor \frac{n}{2} \rfloor}(\pi) - v_{\infty}(\pi) \geq -(1-1/k)^{\lfloor \frac{n}{2} \rfloor-1}$. Combining the latter with relations \eqref{Eq.Lower1} and \eqref{Eq.Lower2} we obtain 
    \begin{align*}
        v_n(q) - v_{\infty}(\pi) \geq - \frac{d_{\lceil \frac{n}{2} \rceil} (M)}{\min_{i\in K} \pi(i)} -\left(1-\frac{1}{k}\right)^{\lfloor \frac{n}{2} \rfloor-1},
    \end{align*}
    as desired. 
    \qed

\section{Discussion and Open Problems}\label{Sec_Discussion}

The results of the current work may give rise to several possible future extensions. Such extensions may concentrate either on the specific model studied in this work, or on generalized models. With regard to the current model, two particular directions seem to be natural. 

\begin{itemize}
    \item \textit{The predictive learning index of homothetic Markov chains.}  A \textit{homothetic} Markov chain is a chain whose transition matrix $M$ depends only on a distribution $\pi \in \dk$, termed the \textit{center}, and a scalar $\beta \in [0,1]$, termed the \textit{ratio}.  Formally,  $M = \beta\cdot \text{Id}_k + (1-\beta) \Pi$, where $\text{Id}_k$ denotes the $k\times k$ identity matrix and $\Pi$ is the matrix whose rows equal to $\pi$. In particular, $\pi$ is an invariant distribution of such $M$, and thus, starting from the prior $\pi$, any homothetic Markov chain with center $\pi$ is stationary. 

For $\beta=0$, a homothetic Markov chain is an i.i.d.\ sequence with distribution $\pi$, whereas for $\beta =1$, the homothetic chain is a deterministic sequence conditional on $X_1$. Therefore, homothetic Markov chains can be interpreted as the following mixture of i.i.d.\ and conditionally deterministic chains: at each step, with probability $\b$ the chain stays put, and with probability $1-\beta$ moves according to $\pi$, irrespective of the preceding state. 

The Markov chain  in Example \ref{Example:1} is homothetic with center $\pi = (0.1,0.2,0.3,0.4)$, and ratio $\beta = 0.36$. Example \ref{Example:1} showcased an instance of what seems to be a more general phenomenon: for the matching utility $u_*$, whenever the Markov chain is homothetic, $i(\pi,M)\leq k$. As such phenomenon was recurring in all computational simulations, we conjecture the following.

\begin{conj}\label{Conj1}
    Let $\{X_j\}_{j\geq 1}$ be a homothetic Markov chain. Then, for the matching utility $u_*$, we have $i(\pi,M)\leq k$ .
\end{conj}
\item \textit{The search for structure in the predictive learning index.} In the concluding remarks of Subsection \ref{Subsec:Main_Index},  we defined the set of matrices  \begin{align*}
    \mathcal{M}^{i\in \N}_{\pi} :=\{M \in \mathcal{M}: i(\pi,M)<+\infty\}.
\end{align*}
Any structural properties of $\mathcal{M}^{i\in \N}_{\pi}$, may contribute to the understanding of sufficient conditions for a Markov chain to have a finite predictive learning index.  

Based on Proposition \ref{Prop.1}, it was established in Subsection \ref{Subsec:Main_Index} that for the matching utility $u_*$, the set $\mathcal{M}^{i\in \N}_{\pi}$  is not convex. Put differently, one cannot always generate a chain with finite index, by taking convex averages of such chains.  A question arises as to whether  $\mathcal{M}^{i\in \N}_{\pi}$ is closed under other algebraic operations, for instance:

\begin{question}\label{Question1}
    Assume that $i(\pi,M)< + \infty$. Is $i(\pi,M^2)< + \infty$ as well?
\end{question}
Such a question may be interpreted in informational terms as well. Must the learning from information obtained from actions along only the odd steps of $\{X_j\}_{j\geq 1}$ be finite, when the learning from all steps is finite. In contrast, assume that the learning along the odd steps is of finite nature; should the learning from all steps be finite as well, or can it be that the information obtained from  the even steps is such that the overall learning is infinite? Formally,

\begin{question}\label{Question2}
    Assume that $i(\pi,M^2)< + \infty$. Is $i(\pi,M)< + \infty$ as well?
\end{question}
Both Conjecture \ref{Conj1} and Questions \ref{Question1} and \ref{Question2} can be viewed as steps toward the following general question:
\begin{question}
    Fix a utility $u$. Does there exist natural sufficient conditions on $\pi$ and $M$ that guarantee that $i(\pi,M)<+\infty$.
\end{question}
\end{itemize}

Another extension may focus on a more general framework, where the information of the decision maker is generated through other filters. Let $\mathcal{S}$ be a finite signal set, and let $\kappa:K\times K \to \Delta(\mathcal{S})$ be the `filter' which assigns to any state $i \in K$ and action $\xi \in K$ a lottery $\kappa(i,\xi)$ over $\mathcal{S}$. At each stage $n\geq 1$, based on the past information, consisting of the actions and signals $(\xi_1,s_1,...,\xi_{n-1},s_{n-1})$, the decision maker chooses an action $\xi_n \in K$, and obtains the signal $s_n$ chosen according to $\kappa(X_n,\xi_n)$.

For any prior $q$, let $v_n(q)$ be the value of the corresponding $n$-stage problem for the utility $u$. Using a suitable Markov decision problem formulation, and applying Lemmas \ref{Convexity Lemma}, \ref{Value along Trajectory Lemma} together with Corollary \ref{Mean-Consistency Property}, which remain valid in the generalized framework, one obtains in the same manner that $\{v_n(\pi):n \geq 1\}$ is non-decreasing in the general case as well.

The rate of convergence of $v_n(\pi)$, the corresponding learning index, and the global behavior of $v_n(\cdot)$, may shed light on the nature of learning under different types of filters.

\appendix
\section{Additional Proofs}

\subsection{Proof of Lemma \ref{Convexity Lemma}}
    Let $q = \a p + (1-\a)z$, where $q,p,z\in \dk$ and $\a \in (0,1)$. We first consider a modified version of $\G$, denoted by $\G'$. At the outset of $\G'$, the decision maker observes the outcome of an `opening lottery' which determines the initial state in $\G$. Such lottery chooses $p$ as the initial state with probability $\a$ and $z$ as the initial state with probability $1-\a$. Upon such a choice of initial state, the the decision maker proceeds to play in $\G$. 
    
    The $n$'th value of $\G'$ (i.e, under the reward function $u_n$) thus equals $\a v_n(p)+(1-\a)v_n(z)$. Nevertheless, as decision maker can play independently of the outcome of the opening lottery in $\G'$, since the initial state in $\G'$ has expectation $q$, by the definition of $\G'$ she can guarantee $v_n(q)$. We thus conclude that $v_n(q)  \leq \a v_n(p) + (1-\a)v_n(z)$, as required.   
\qed 

\subsection{Proof of Lemma \ref{Beliefs Exp. Lemma}}
    Our proof follows an inductive argument on $j\geq t$. For $j=t$ the result holds trivially. For the induction step assume that the claim holds for $t,...,n$ (where $j\geq t$), and let us show that it holds for $j+1$ as well. We have
    \begin{align*}
         & E_{q,\s}\, (q_{j+1} \mid D) = \frac{1}{P_{q,\s}(D)}\cdot E_{q,\s}\, (q_{j+1} \cdot \textbf{1}\{D\})\\
         & = \frac{1}{P_{q,\s}(D)} \cdot E_{q,\s}\left(E_{q,\s}\, (q_{j+1} \cdot \textbf{1}\{D\}\mid q_1,\xi_1,..., q_{j},\xi_{j})\right) \nonumber\\
         & = \frac{1}{P_{q,\s}(D)} \cdot E_{q,\s}\left(\textbf{1}\{D\} \cdot E_{q,\s}\, (q_{j+1} \mid q_1,\xi_1,..., q_{j},\xi_{j})\right) \nonumber\\
         & = E_{q,\s}\left(   \sum_{i=1}^k \xi_{j}(i) \left[ q_{j}(i) \cdot m_i^* + (1-q_{j}(i)) \cdot w(q_{j},i)M\right] \Bigg|\, D \right)\nonumber\\
         & = E_{q,\s}\left(   \sum_{i=1}^k \xi_{j}(i) \left[ q_{j}(i) \cdot m_i + q_{j} M - q_{j}(i)\cdot \delta_{\{i\}}M \right]  \Bigg| \, D \right)\\
         & = E_{q,\s}\left(   \left(\sum_{i=1}^k \xi_{j}(i) \right) q_{j}M \, \Bigg|\, D\right) = E_{q,\s}\left(q_{j}M \mid D\right) \nonumber\\
         & = E_{q,\s}\left(q_{j}\mid D\right)M = E_{q,\s}\, (q_t \mid D)M^{j-t}M = E_{q,\s}\, (q_t \mid D) M^{j+1-t}, \nonumber
    \end{align*}
where the fourth equality follows from the definition of the transition rule in $\G$,  the fifth equality follows from the fact that by definition $w(q_{j},i) := (q_{j} - q_{j}(i)\delta_{\{i\}})/(1-q_{j}(i))$, the sixth from the fact that $\delta_{\{i\}}M = m_i$, the seventh from the linearity of the expectation operator, and the second to last equality is due to the induction hypothesis. 
\qed 

\subsection{Proof of Lemma \ref{DT_Lemma}}
    Fix $\s \in \S^*$. The proof proceeds by induction of $j$. For $j=1$ we have
    \begin{align*}
        u_1(\a q +(1-\a)p,\s) & = \sum_{i \in K} (\a q +(1-\a)p)(i)u(i,\s(0)) \\
        & = \a \sum_{i \in K} q(i)u(i,\s(0)) + (1-\a)\sum_{i \in K} p(i)u(i,\s(0))\\
        & = \a u_1(q,\s) + (1-\a) u_1(p,\s),
    \end{align*}
    proving the basis of the induction. 
    For the induction step, define the strategies $\s_{00},\s_{01} \in \S^*$ by $\s_{0a}(x) := \s(0,a,x)$ for every finite binary string $x$, where $a \in \{0,1\}$. Let $i = \s(0)$ be the first action $\s$ takes. Also, let $\theta:= \a q +(1-\a)p$. We have:
    \begin{align}\label{DT_Lemma_eq1}
        u_n(\theta,\s) & = \theta(i) u_{n-1}(m^*_i,\s_{01})+ (1-\theta(i))u_{n-1}(w(\theta,i)M,\s_{00}).
    \end{align}
    As it can be verified that
    \begin{align*}
            w(\theta,i) = \frac{\a(1-q(i))}{1- \theta(i)}\cdot  w(q,i) + \frac{(1-\a)(1-p(i))}{1- \theta(i)}\cdot  w(p,i),
    \end{align*}
    we may apply the induction hypothesis for $n-1$ to obtain
    \begin{align*}
      (1-\theta(i))u_{n-1}(w(\theta,i)M,\s_{00}) & =   \a(1-q(i)) u_{n-1}(w(q,i)M,\s_{00}) \\
      & \quad + (1-\a)(1-p(i))u_{n-1}(w(p,i)M,\s_{00}). 
    \end{align*}
    Plugging the latter back to relation \eqref{DT_Lemma_eq1} and rearranging the terms we obtain:
    \begin{align*}
        u_n(\theta,\s) & = \a \left[ q(i) u_{n-1}(m^*_i,\s_{01}) + (1-q(i)) u_{n-1}(w(q,i)M,\s_{00}) \right] \\
        & + \quad (1-\a) \left[ p(i) u_{n-1}(m^*_i,\s_{01}) + (1-p(i)) u_{n-1}(w(p,i)M,\s_{00}) \right]\\
        & = \a u_n(q,\s) + (1-\a)u_n (p,\s),
    \end{align*}
    completing the induction step, and thus the proof of the lemma. 
\qed
 \end{document}